\documentclass{amsart}
\usepackage{amssymb,amsmath,amsthm,amsfonts}
\usepackage{graphicx}
\usepackage{tikz}
\usetikzlibrary{arrows.meta,positioning}
\usepackage{environ}
\usepackage{float}
\usepackage{xcolor}    

\usepackage{enumitem} 
\def\lf{\left}
\def\r{\right}
\def\ls{\lesssim}

\newcommand{\field}[1]{\mathbb{#1}}
\newcommand{\N}{\field{N}}                      
\newcommand{\Z}{\field{Z}}                      
\newcommand{\R}{\field{R}}                      

\NewEnviron{bigequation}{
    \begin{equation}
    \scalebox{1.5}{$\BODY$}
    \end{equation}
    }

\newcommand{\de}{\delta}
\newcommand{\la}{\lambda}
\newcommand{\eps}{\varepsilon}

\newcommand{\loc}{{\scriptstyle{loc}}}
\newcommand{\ovdimB}{{\overline{\dim_B}}}

\def\Barint_#1{\mathchoice
          {\mathop{\vrule width 6pt height 3 pt depth -2.5pt
                  \kern -8pt \intop}\nolimits_{#1}}%
          {\mathop{\vrule width 5pt height 3 pt depth -2.6pt
                  \kern -6pt \intop}\nolimits_{#1}}%
          {\mathop{\vrule width 5pt height 3 pt depth -2.6pt
                  \kern -6pt \intop}\nolimits_{#1}}%
          {\mathop{\vrule width 5pt height 3 pt depth -2.6pt
                  \kern -6pt \intop}\nolimits_{#1}}}

\makeatletter
\newcommand*{\mint}[1]{%
  \mint@l{#1}{}%
}
\newcommand*{\mint@l}[2]{%
  \@ifnextchar\limits{%
    \mint@l{#1}%
  }{%
    \@ifnextchar\nolimits{%
      \mint@l{#1}%
    }{%
      \@ifnextchar\displaylimits{%
        \mint@l{#1}%
      }{%
        \mint@s{#2}{#1}%
      }%
    }%
  }%
}
\newcommand*{\mint@s}[2]{%
  \@ifnextchar_{%
    \mint@sub{#1}{#2}%
  }{%
    \@ifnextchar^{%
      \mint@sup{#1}{#2}%
    }{%
      \mint@{#1}{#2}{}{}%
    }%
  }%
}
\def\mint@sub#1#2_#3{%
  \@ifnextchar^{%
    \mint@sub@sup{#1}{#2}{#3}%
  }{%
    \mint@{#1}{#2}{#3}{}%
  }%
}
\def\mint@sup#1#2^#3{%
  \@ifnextchar_{%
    \mint@sup@sub{#1}{#2}{#3}%
  }{%
    \mint@{#1}{#2}{}{#3}%
  }%
}
\def\mint@sub@sup#1#2#3^#4{%
  \mint@{#1}{#2}{#3}{#4}%
}
\def\mint@sup@sub#1#2#3_#4{%
  \mint@{#1}{#2}{#4}{#3}%
}
\newcommand*{\mint@}[4]{%
  \mathop{}%
  \mkern-\thinmuskip
  \mathchoice{%
    \mint@@{#1}{#2}{#3}{#4}%
        \displaystyle\textstyle\scriptstyle
  }{%
    \mint@@{#1}{#2}{#3}{#4}%
        \textstyle\scriptstyle\scriptstyle
  }{%
    \mint@@{#1}{#2}{#3}{#4}%
        \scriptstyle\scriptscriptstyle\scriptscriptstyle
  }{%
    \mint@@{#1}{#2}{#3}{#4}%
        \scriptscriptstyle\scriptscriptstyle\scriptscriptstyle
  }%
  \mkern-\thinmuskip
  \int#1%
  \ifx\\#3\\\else_{#3}\fi
  \ifx\\#4\\\else^{#4}\fi  
}
\newcommand*{\mint@@}[7]{%
  \begingroup
    \sbox0{$#5\int\m@th$}%
    \sbox2{$#5\int_{}\m@th$}%
    \dimen2=\wd0 %
    \let\mint@limits=#1\relax
    \ifx\mint@limits\relax
      \sbox4{$#5\int_{\kern1sp}^{\kern1sp}\m@th$}%
      \ifdim\wd4>\wd2 %
        \let\mint@limits=\nolimits
      \else
        \let\mint@limits=\limits
      \fi
    \fi
    \ifx\mint@limits\displaylimits
      \ifx#5\displaystyle
        \let\mint@limits=\limits
      \fi
    \fi
    \ifx\mint@limits\limits
      \sbox0{$#7#3\m@th$}%
      \sbox2{$#7#4\m@th$}%
      \ifdim\wd0>\dimen2 %
        \dimen2=\wd0 %
      \fi
      \ifdim\wd2>\dimen2 %
        \dimen2=\wd2 %
      \fi
    \fi
    \rlap{%
      $#5%
        \vcenter{%
          \hbox to\dimen2{%
            \hss
            $#6{#2}\m@th$%
            \hss
          }%
        }%
      $%
    }%
  \endgroup
}

\theoremstyle{definition}
\newtheorem{theorem}{Theorem}
\newtheorem{theoremA}{Theorem}
\newtheorem{corollary}[theorem]{Corollary}
\newtheorem{lemma}[theorem]{Lemma}

\newtheorem{proposition}[theorem]{Proposition}

\theoremstyle{definition}
\newtheorem{definition}[theorem]{Definition}

\newtheorem{remark}[theorem]{Remark}

\numberwithin{theorem}{section} \numberwithin{equation}{section}

\newenvironment{claimproof}{%
  \begin{proof}
}{\end{proof}}

\title[Dimension distortion under fractionally smooth mappings]{On the dimension distortion under fractionally smooth mappings}
\begin{document}
\author[Ryan Alvarado]{Ryan Alvarado}
\address{Department of Mathematics \\ Amherst College, Amherst, MA}
\email{rjalvarado@amherst.edu}
\author[Efstathios-K. Chrontsios-Garitsis]{Efstathios-K. Chrontsios-Garitsis}
\address{Department of Mathematics \\ University of Tennessee, Knoxville, TN}
\email{echronts@utk.edu, echronts@gmail.com}

\begin{abstract}
We determine the extent to which certain classes of fractionally  `smooth' continuous mappings between metric spaces distort various dimensions, including the Hausdorff, upper Minkowski (box-counting), and upper intermediate dimensions. Our intermediate and Minkowski dimension distortion results are new even for continuous (fractional) Sobolev and, more generally, Triebel--Lizorkin and Besov mappings
between Euclidean spaces, complementing 
the work of Hencl-Honz\'ik \cite{SobSubcriticalDimDist2, OLDERdimHfractionalSobolevTL} and Huynh \cite{Chi_thesis}. 
Moreover, our results also extend the aforementioned work, as well as the work of Kaufman \cite{Kaufman} and Fraser-Tyson \cite{FraTyson_intermed_dim}, to certain weighted Euclidean spaces and, more generally, to doubling metric measure spaces. As an application of our main result, we quantify the corresponding dimension distortion properties of quasisymmetric mappings for non-Ahlfors regular subsets of metric measure spaces, strengthening a result of Bishop-Hakobyan-Williams \cite{BishHakWill:QSdistortion}.
\end{abstract}

\maketitle
\section{Introduction}
The class of Sobolev mappings has been famously an essential tool in the area of partial differential equations (PDEs) 
\cite{EvansPDE}. On the other hand, fractionally smooth Sobolev spaces provide a natural framework for problems where smoothness is intermediate between integer orders, acting as interpolation spaces in the context of functional analysis  \cite{Fract_Sob_book}. They arise naturally in the study of the fractional Laplacian, with applications on anomalous diffusion and jump processes, minimal surfaces, elliptic problems with measure data, and many other areas. We refer to \cite{Fract_Sob_applications} for an even more extensive list of applications and relevant references.  In the past two and a half decades, there has been an increasing interest and need to extend this theory of (fractional) Sobolev and, more generally, Triebel--Lizorkin and Besov mappings to metric spaces. Applications of this endeavor include the development of the theory of  PDEs \cite{KigamiAnFr}, calculus of variations \cite{AmbrosioBV} and optimal transportation \cite{AmbrosioGradFlow} on the non-smooth setting of fractal spaces. The theory of  (fractional) Sobolev, Triebel--Lizorkin, and Besov mappings defined between metric spaces has been developed by many authors (see, for instance, \cite{CheegerSob,HajSob,HajKoskSob,HK00,KorSchSob,NagesNewtonianSob,Y03,KYZ11,hhhpl21,WHYH21,GKZ13}
for a non-exhaustive list), who have used different approaches to adjust the theory to different settings. 

A question of broad interest has been to determine in what ways certain classes of mappings distort  dimension notions. One of the earliest results in this direction is by Gehring-V\"ais\"al\"a \cite{GehringVais}, who gave quantitative bounds on how quasiconformal mappings, a special class of super-critical Sobolev mappings, change the Hausdorff dimension of a subset of $\R^n$. Kaufman later proved  bounds for the distortion of the Hausdorff and Minkowski dimensions under general super-critical Sobolev mappings \cite{Kaufman}. The study of dimension distortion has since  been extended to sub-critical Sobolev mappings \cite{SobSubcriticalDimDist1}, fractionally smooth Sobolev mappings \cite{SobSubcriticalDimDist2,OLDERdimHfractionalSobolevTL, Chi_thesis}, to other dimension notions \cite{OurQCspec,SobSubcritDimDist0, HolomSpecChron, FraTyson_intermed_dim}, and to other settings, such as distortion by Sobolev and quasisymmetric mappings defined between metric spaces  \cite{btw:heisenberg}, \cite{bmt:grassmannian}, \cite{BaloghAGMS}, \cite{BishHakWill:QSdistortion}. Moreover,  employing dimension distortion results has yielded applications to the H\"older and quasiconformal classification problems \cite{Chron_conc_shells, OurQCspec, CG_Vellis_spirals}, which are often challenging tasks and require the use of heavy machinery. However, in the non-Euclidean setting all results for Sobolev mappings are for the Hausdorff dimension, or cases where all dimensions coincide, with the distortion of the Minkowski dimension only recently settled by the second author \cite{Chron_comp_holder_Minkowski}. Furthermore, there are currently no results on the dimension distortion under fractional Sobolev, Triebel--Lizorkin, or Besov mappings in the metric setting. In this manuscript, we address both of these open directions by considering the dimension distortion properties of \textit{compactly-H\"older mappings}, a class that contains and unifies the notions of  (fractional) Sobolev, Triebel--Lizorkin, and Besov mappings between metric spaces (see Section~\ref{sec:background}).

A modern approach to notions of fractal dimensions has been the introduction of dimension functions, which provide more information on the finer structure of spaces than typical dimension values. For instance, the Assouad spectrum, introduced by Fraser-Yu \cite{fy:assouad-spectrum}, is such a dimension function that naturally interpolates between the upper Minkowski and Assouad dimensions, based on the geometric properties of the space. We refer to \cite{FraserBook} for an exposition on the topic, with a plethora of applications in areas such as number theory, probability, and functional analysis.  A similar dimension function is the collection of \textit{(upper) intermediate dimensions}, introduced by Falconer-Fraser-Kempton \cite{Intermediate_dim_introduction}. This notion  interpolates between the Hausdorff and upper Minkowski dimensions, capturing finer geometric traits of the space not typically distinguished by the two extreme dimension values. Recent applications include towards the dimension theory of Brownian images \cite{Falc_Brown_applic_interm}, bi-Lipschitz classification of spaces \cite{banaji_Advanc_Lipsch}, and the theory of orthogonal projections \cite{Fra_proj_interm}. For a uniformly perfect metric space $X$ (see Section~\ref{sec:background}), the properties of $\theta$-intermediate dimensions of non-empty subsets $E$ of $X$, denoted by $\dim_\theta E$, were first established and studied by Banaji in \cite{Banaji_Int_dim_metric}. In the same metric context, we establish intermediate dimension distortion bounds for compactly H\"older mappings, in the spirit of Gehring-V\"ais\"al\"a \cite{GehringVais} and Kaufman \cite{Kaufman}:

\begin{theorem}\label{thm:main_Holder}
    Suppose $(X,d_X)$ is a doubling, uniformly perfect metric space and $(Y,d_Y)$ is a uniformly perfect metric space. For $\theta\in(0,1)$, $p\in(0,\infty)$ and $\alpha\in(0,\infty)$, if $f:X\rightarrow Y$ is $(p,\alpha)$-compactly H\"older and $E\subset X$ is bounded with $\dim_\theta E =d_E(\theta)$, then
		\begin{equation}\label{eq:CHBoxdistortion}
			\dim_\theta f(E) \leq \max\left\{ \frac{pd_E(\theta)}{\alpha p+d_E(\theta)}, d_E(\theta)\right\}.
		\end{equation}
\end{theorem}

An immediate corollary is a similar dimension bound under Newtonian and quasisymmetric mappings. In particular, it follows by \cite[Theorem~1.2]{Chron_comp_holder_Minkowski} that a continuous mapping in the Newtonian-Sobolev class is also  compactly H\"older  for appropriate constants $p$ and $\alpha$, which also yields a similar inclusion for quasisymmetric mappings, under standard assumptions on $X$ and $Y$ (see \cite[Corollary~1.3]{Chron_comp_holder_Minkowski}).

\newpage
\begin{corollary}\label{cor:QS_Dim}
    ~
    \begin{enumerate}
        \item[(i)] Suppose $(X,d,\mu)$ is a proper, locally $Q$-homogeneous 
        %
    metric measure space supporting a $Q$-Poincar\'e inequality for some $Q\in(1,\infty)$, and $(Y,d_Y)$ is an arbitrary uniformly perfect metric space.
    Let $f:X\rightarrow Y$ be a continuous mapping with an upper gradient  $g\in L^p_\loc(X)$  with $p\in(Q,\infty)$. If $\theta\in(0,1)$ and $E\subset X$ is bounded with $\dim_\theta E =d_E(\theta)<Q$, then
		\begin{equation}
			\dim_\theta f(E) \leq \frac{p d_E(\theta)}{p-Q+d_E(\theta)}<Q.
		\end{equation}
        \item[(ii)] Suppose $Q\in(1,\infty)$ and $(X,d,\mu)$ is a proper and  $Q$-Ahlfors regular metric measure space supporting a $p_0$-PI for $p_0\in(1,Q)$, and $(Y,d_Y)$ is a $Q$-Ahlfors regular metric space. Let $f:X\rightarrow Y$ be an $\eta$-quasisymmetric homeomorphism. If $\theta\in(0,1)$ and $E\subset X$ is bounded with $\dim_\theta E =d_E(\theta)\in (0,Q)$, then
		\begin{equation}\label{eq:QSdistortion}
			0< \frac{(p-Q)d_E(\theta)}{p-d_E(\theta)}  \leq \dim_\theta f(E) \leq \frac{p d_E(\theta)}{p-Q+d_E(\theta)}<Q,
		\end{equation} where $p>Q$ only depends on $\eta(1), \eta^{-1}(1)$.
    \end{enumerate}
\end{corollary}
The Newtonian-Sobolev class constitutes one of the broader classes of Sobolev-type mappings between metric spaces (see Theorem~10.5.1 in \cite{Juha-Jeremy-etc-book}). Hence, the above corollary is a broad generalization of the result of Fraser-Tyson \cite{FraTyson_intermed_dim} and settles the intermediate dimension distortion problem on metric spaces, by providing a quantitative bound. 
A non-exhaustive list of spaces where the above result could be applied includes Carnot groups, Laakso spaces, Gromov hyperbolic groups and boundaries (see \cite{Juha-Jeremy-etc-book} Chapter 14  and references therein). 

Bishop-Hakobyan-Williams \cite{BishHakWill:QSdistortion} studied the quasisymmetric dimension distortion problem   in the case where the input set $E$ is Ahlfors regular, which implies that all dimensions of $E$ coincide. Their motivation was the absolute continuity on lines  property (ACL) that quasisymmetric mappings satisfy in the Euclidean setting. Their result provides a fundamental generalization of this fact in the metric measure spaces setting. In general, however, the Hausdorff, intermediate, and Minkowski dimensions of $E$ could all differ. In such a case, the results from \cite{BishHakWill:QSdistortion} cannot be applied, while \eqref{eq:QSdistortion} provides quantitative bounds on $\dim_\theta f(E)$, which also recover \cite[Theorem~1.2]{Chron_comp_holder_Minkowski}.

In order to reach the desired dimension distortion statement for (fractional) Sobolev, Triebel--Lizorkin, and Besov  mappings, we need to ensure appropriate inclusions in the compactly H\"older class. This is achieved through  embedding results of the metric Sobolev mappings of fractional smoothness in question, which generalize the first author's work on real-valued functions  \cite{agh20,AYY24,AYY21}  (see Section 3.2).

\begin{theorem}\label{thm:FractionalIntDim}
	Let $(X,d_X,\mu)$ be a proper locally $Q$-homogeneous
    metric measure space and $(Y,d_Y)$ is an  arbitrary metric space. Let $s\in(0,\infty)$, $p\in (Q/s,\infty)$, $q\in(0,\infty]$ and $f:X\rightarrow Y$ be a continuous mapping. If $f$
    has  a finite Haj\l{}asz--Triebel--Lizorkin {semi-norm}  $\Vert f\Vert_{\dot{M}^s_{p,q}(X:Y)}$, or a finite Haj\l asz--Besov {semi-norm} $\Vert f\Vert_{\dot{N}^s_{p,q}(X:Y)}$ with $q\leq p$, then $f$ is $(t,s-Q/t)$-compactly H\"older for all $t\in (Q/s,p)$. If  $\Vert f\Vert_{\dot{N}^s_{p,q}(X:Y)}<\infty$ for $p<q<\infty$, then $f$ is $(q,s-Q/t)$-compactly H\"older for all $t\in (Q/s,p)$.
    If, in addition to the assumptions above,  $X$ and $Y$ are uniformly perfect, then the following statements hold for  any $\theta\in(0,1)$ and any bounded set $E\subset X$ with $\dim_\theta E =d_E(\theta)<Q$.
    \begin{enumerate}
        \item[(i)]If $\Vert f\Vert_{\dot{M}^s_{p,q}(X:Y)}<\infty$, or if $\Vert f\Vert_{\dot{N}^s_{p,q}(X:Y)}<\infty$ with $q\leq p$, then 
        \begin{equation}\label{eq: main TL-Besov q<=p Interm bound}
		    \dim_\theta f(E)\leq \max\left\{
            \displaystyle\frac{p \, d_E(\theta)}{s p - Q + d_E(\theta)},\; d_E(\theta)
            \right\}.
		\end{equation}
        \item [(ii)] If $\Vert f\Vert_{\dot{N}^s_{p,q}(X:Y)}<\infty$ with $0<p<q$, then
        \begin{equation}\label{eq:SobBoxdistortion}
            \dim_\theta f(E)\leq \max\left\{\frac{q \, d_E(\theta)}{(s - Q/p) q + d_E(\theta)}, d_E(\theta)\right\}.
        \end{equation}
    \end{enumerate}
\end{theorem}

Theorem~\ref{thm:FractionalIntDim} is in fact novel already in the Euclidean case $X=Y=\R^n$ with the usual metric and measure, as it is the first intermediate dimension distortion result for classical continuous Triebel--Lizorkin and Besov mappings, which are known to coincide with the Haj\l{}asz--Triebel--Lizorkin and Haj\l{}asz--Besov mappings in this setting (see Subsection~\ref{subsec:BackSobolev}). In particular, this yields another Euclidean novel dimension distortion result, this time for  classical fractional Sobolev mappings. In addition, Theorem~\ref{thm:FractionalIntDim} also immediately yields 
dimension distortion bounds for continuous mappings whose
fractional Haj\l{}asz--Sobolev semi-norm 
$\Vert \cdot\Vert_{\dot{M}^{s,p}(X:Y)}$ is finite, since 
$\dot{M}^s_{p,\infty}(X:Y)=\dot{M}^{s,p}(X:Y)$; 
see Lemma~\ref{sobequal}.  Furthermore, if $X, Y$ are not uniformly perfect, similar and simplified arguments that lead to Theorem~\ref{thm:FractionalIntDim} can be employed for $\dim_0E=\dim_H E$. Consequently,  the work of Kaufman on Euclidean super-critical Sobolev mappings \cite{Kaufman}, the work of Hencl-Honz\'ik on Euclidean Triebel-Lizorkin mappings 
\cite{SobSubcriticalDimDist2}, and the work of Huynh on Euclidean Besov mappings 
\cite{Chi_thesis} are also recovered. In particular, we have the following unifying result for both the Hausdorff and the Minkowski dimensions.

\begin{theorem}\label{thm: dimH dimB}
Suppose $(X,d_X,\mu)$ is a proper, locally $Q$-homogeneous metric measure space, $(Y,d_Y)$ is an  arbitrary metric space, and let $s$, $p$, $q$, and $f$ be as in Theorem~\ref{thm:FractionalIntDim}.
    \begin{enumerate}
        \item[(i)] If $E\subset X$ with $\dim_H E=a$, then
        \begin{equation}\label{eq: dim_H bound}
        \dim_H f(E) \leq
\begin{cases}
\displaystyle\max\left\{\frac{q \, a}{(s - Q/p) q + a}, a\right\}, &  p < q<\infty \,\, \&\,\, \Vert f\Vert_{\dot{N}^s_{p,q}(X:Y)}<\infty, \\[3ex]
\max\left\{
\displaystyle\frac{p \, a}{s p - Q + a},\; a
\right\}, & p \geq q \,\, \&\,\, \Vert f\Vert_{\dot{N}^s_{p,q}(X:Y)}<\infty,\\[3ex]
\max\left\{
\displaystyle\frac{p \, a}{s p - Q + a},\; a
\right\}, & \Vert f\Vert_{\dot{M}^s_{p,q}(X:Y)}<\infty.
\end{cases}
        \end{equation}
        \item[(ii)] If $E\subset X$ is  bounded with $\dim_B E=a$, then \eqref{eq: dim_H bound} holds with $\dim_H f(E)$ replaced by $\dim_B f(E)$.
    \end{enumerate}
\end{theorem}

Although several approaches to defining Sobolev, Triebel--Lizorkin, and Besov spaces on metric spaces have been proposed over the years, a key advantage of the spaces considered in this work is that they support a robust theory without requiring the underlying metric space to be connected (e.g.\ Cantor-type sets where other Sobolev notions are trivial), or the measure to be (globally) doubling. Under additional assumptions on the metric measure space (e.g.\ Ahlfors regularity, Poincar\'e inequality), many of the proposed definitions are known to coincide and, hence, Theorems~\ref{thm:FractionalIntDim} and \ref{thm: dimH dimB} can be used to provide dimension distortion bounds for an even  broader class of Sobolev, Triebel--Lizorkin, and Besov spaces in that setting. We refer the interested reader to \cite{AWYY21,KYZ11,GKZ13} and the references therein for more details. It should also be pointed out that we in fact prove the dimension distortion results of Theorems~\ref{thm:FractionalIntDim}, \ref{thm: dimH dimB} for mappings with \textit{locally} finite semi-norms (see Remark~\ref{rem: local true}).

Note that the bounds involving the Hausdorff and Minkowski dimensions in Theorem~\ref{thm: dimH dimB} are new for weighted Euclidean spaces, extending the respective results from \cite{SobSubcriticalDimDist2, OLDERdimHfractionalSobolevTL, Chi_thesis} in that popular setting. For instance, if $\la_n$ is the $n$-Lebesgue measure, the conditions of Theorem~\ref{thm:FractionalIntDim} (and thus of Theorem~\ref{thm: dimH dimB}) are satisfied by the weighted Euclidean metric measure space $(\R^n,d_{\text{euc}}, w\la_n)$ for a wide variety of weights $w:\R^n\rightarrow [0,\infty]$, such as the class of Muckenhoupt weights (see Chapter 1 in \cite{HeinonenWeightedSobBook}). These weights were introduced by Muckenhoupt \cite{MuckOrigin} in order to characterize the boundedness of the Hardy-Littlewood maximal operator on weighted $L^p$ spaces, and have since established an active area within functional and harmonic analysis \cite{GrafClassMucken, GrafModMucken}. More recently, certain Muckenhoupt weights have been associated with dimension theoretic characteristics of spaces through the ``weak porosity'' notion (see \cite{CarlosEtAlWeakPoros, CarlosWeakPorosMetric}).

This paper is organized as follows. 
Section~\ref{sec:background} includes the necessary background for relevant mapping classes and dimensions, and introduces an equivalent definition of the intermediate dimension using Hyt\"onen--Kairema's dyadic cubes.
In Section~\ref{sec:DimDistProofs} we prove the intermediate dimension distortion result under compactly H\"older maps (Theorem~\ref{thm:main_Holder}). Section~\ref{sec: Sob are cH} contains the proofs of Morrey-type embedding theorems, which we employ to show the Sobolev classes of fractional smoothness in question are contained in the appropriate compactly H\"older class, and prove Theorem~\ref{thm:FractionalIntDim}. Section~\ref{sec: dimH et al} contains the proof for the Hausdorff and Minkowski dimension distortion under Haj\l{}asz Triebel--Lizorkin  and Haj\l{}asz Besov mappings (Theorem~\ref{thm: dimH dimB}), as well as the proof of Corollary~\ref{cor:QS_Dim} for the distortion under Newtonian Sobolev and quasisymmetric mappings.
\medskip

\paragraph{\bf Acknowledgments.} Part of the project was completed during the first author's visit to Knoxville in Spring
2025 and, thus, he wishes to thank the University of Tennessee for their hospitality.

\section{Background}\label{sec:background}

\subsection{Metric spaces and dimensions.}\label{subsec:MS}
Given two non-negative quantities $A$ and $B$, we write $A \lesssim B$ if there is a comparability constant $C = C(\lesssim)>0$ such that $A \leq C B$. Similarly, we write $A \gtrsim B$ if there is $C = C(\gtrsim)$ such that $A \geq B/C$. If $A \lesssim B$ and $A \gtrsim B$ we write $A \simeq B$.

Let $(X,d_X)$ be a metric space. We often omit the subscript and write $d(x,y)$ for $x,y\in X$ if the space is understood. We denote the open ball centered at $x$ of radius $r>0$ by $B_X(x,r):= \{z\in X: \,\, d(x,z)<r \}.
$ If the space is understood, we often omit the subscript $X$. Given a ball $B=B(x,r) \subset X$, we  denote by $\lambda B$ the ball $B(x,\lambda r)$, for $\lambda>0$. Given a non-empty set $U\subset X$, we denote by $|U|$ the diameter of $U$ in the metric of $X$. We also make the convention that all bounded sets we consider henceforth are non-empty, even if not explicitly stated, as all results trivially follow otherwise.

We say that $(X,d)$ is a \textit{doubling metric space} if there is a \textit{doubling constant} $C_d\geq 1$ such that for all $x\in X, r>0$, the smallest number of balls of radius $r$ needed to cover $B(x,2r)$ is at most $C_d$. Note that the doubling property implies that $X$ is separable.

We say that $(X,d)$ is a \textit{uniformly perfect} metric space if there is $c_u\in (0,1)$ such that for every $x\in X$ and every $r<|X|$ there is a point $x'\in B(x,r)\setminus B(x,c_ur)$. We say that $(X,d)$ is a $c_u$-uniformly perfect metric space if we need to emphasize the constant.

Let $E$ be a bounded subset of $X$. For $r>0$, denote by $N(E,r)$ the smallest number of sets of diameter at most $r$ needed to cover $E$. The {\it (upper) Minkowski dimension} of $E$ is defined as
$$
\ovdimB(E) = \limsup_{r\to 0} \frac{\log N(E,r)}{\log(1/r)}.
$$
This notion is also known as \textit{upper box-counting dimension}, which justifies the notation with the subscript `B' typically used in the literature (see \cite{FalcBook}, \cite{FraserBook}). We drop the adjective `upper' and the bar notation throughout this paper as we will make no reference to the lower Minkowski dimension.  If $X$ is a typical Euclidean space $\R^n$, an equivalent formulation for the Minkowski dimension is the following (see \cite[Definition 1.1]{Intermediate_dim_introduction})
$$
\dim_B E
= \inf \left\{\, d>0:\;
  \vcenter{\hbox{$
    \begin{aligned}[t]
    \\[-0.6ex] 
    & \forall\,\varepsilon>0 \;\exists\,\delta_\varepsilon\in(0,1)\ \text{such that}\;\forall\,\delta\in(0,\delta_\varepsilon)\; \text{there is} \\[0.3ex]
    & \{U_i\}_i\ \text{cover of }E
       \ \text{with }\; |U_i|=\delta, \; \forall i,\; \text{and} \;\sum_i |U_i|^{d}<\varepsilon
    \end{aligned}
  $}}
\right\}.
$$On the other hand, if $|U_i|=\delta$ is replaced by $|U_i|\leq \delta$, the above definition would yield the Hausdorff dimension $\dim_H E$ of the set $E$. These representations motivated Falconer, Fraser, and Kempton to define the notion of intermediate dimensions \cite{Intermediate_dim_introduction}, which is a dimension function, rather than a dimension value,  geometrically interpolating between $\dim_H E$ and $\dim_B E$. While it was initially defined in \cite{Intermediate_dim_introduction} for the Euclidean setting, we state the definition  for subsets $E$ of a uniformly perfect metric space $X$. For $\theta, \delta\in(0,1)$ we say that a cover $\{U_i\}_{i\in I}$ of $E$ is \textit{$\delta^{1/\theta}$-admissible} if $\delta^{1/\theta}\leq |U_i|\leq \delta$, for all $i\in I$. The \textit{($\theta$-upper-)intermediate dimension} of $E$ is defined to be
$$
\dim_\theta E
= \inf \left\{\, d>0:\;
  \vcenter{\hbox{$
    \begin{aligned}[t]
    \\[-0.6ex] 
    & \forall\,\varepsilon>0 \;\exists\,\delta_\varepsilon\in(0,1)\ \text{such that}\;\forall\,\delta\in(0,\delta_\varepsilon) \; \text{there is}\, \\[0.3ex]
    &  \{U_i\}_{i\in I}\;\delta^{1/\theta}\text{-admissible cover of }E
       \ \text{with }\; \;\sum_{i\in I} |U_i|^{d}<\varepsilon
    \end{aligned}
  $}}
\right\}.
$$Similarly to the Minkowski dimension convention, we make no mention to the lower intermediate dimension for the rest of the paper and, hence, we drop the adjective `upper' in this case too. It should also be noted that in the uniformly perfect metric setting, if $c_1, c_2>0$ are fixed constants and we slightly modify the notion $\delta$-admissible cover to require $c_1\delta\leq |U_i|\leq c_1 \delta$ instead of equality, we similarly have
$$
\dim_H E=\dim_0 E \qquad \text{and}\qquad \dim_B E=\dim_1 E,
$$ with $\dim_\theta E$ being a continuous function of $\theta$ in $(0,1]$. We refer to \cite{Banaji_Int_dim_metric} for a very interesting treatment of this, and other similar notions (generalized intermediate dimensions) in the metric setting.

On Euclidean spaces $X=\R^n$ with the usual metric one can use dyadic cubes instead of arbitrary sets of diameter between $\delta^{1/\theta}$ and $\delta$ to define the intermediate dimension (see \cite{FraTyson_intermed_dim}), and similarly for other dimension notions (see \cite{FalcBook, FraserBook}). On arbitrary metric spaces, however, there are various generalizations of dyadic cube constructions. One of the first manuscripts addressing this idea was by David  \cite{CubesGuyC1}, while one of the first explicit constructions of a system of dyadic cubes is due to Christ \cite{ChristCubes}. See also \cite{MoreCubes1}, \cite{Hyt:dyadic}, \cite{MoreCubes2}, \cite{MoreCubes3}, which is not an exhaustive list. The most fitting  notion for our context is that of Hyt\"onen and Kairema, which is enough to characterize various notions of dimension, including the Hausdorff, Assouad \cite{chron_metric_dyadic_dim} and the Minkowski \cite{Chron_comp_holder_Minkowski} dimensions.
\begin{theoremA}[Hyt\"onen, Kairema \cite{Hyt:dyadic}]\label{thm:Dydadic}
	Suppose $(X,d)$ is a doubling metric space. Let $0<c_0\leq C_0<\infty$ and $b\in (0,1)$ with $12 C_0 b \leq c_0$. For any non-negative $k\in \Z$ and collection of points $\{ z_i^k \}_{i\in I_k}$ with
	
	\begin{equation}\label{eq:centers_away}
		d(z_i^k,z_j^k)\geq c_0 b^k, \,\,\,\, \text{for} \,\, i\neq j
	\end{equation}
	and
	\begin{equation}\label{eq:points_close_centers}
	\min_i d(z_i^k,x)< C_0 b^k, \,\,\,\, \text{for all} \,\, x\in X
	\end{equation}
	we can construct a collection of sets $\{ Q_i^k \}_{i\in I_k}$ such that
	\begin{itemize}
		\item[(i)] if $l \geq k$ then for any $i\in I_k$, $j\in I_l$ either $Q_j^l\subset Q_i^k$ or $Q_j^l \cap Q_i^k=\emptyset$,
		\vspace{0.1cm}
		\item[(ii)] $X$ is equal to the disjoint union $\bigcup\limits_{i\in I_k} Q_i^k$, for every $k\in\N$
		\vspace{0.1cm}
		\item [(iii)] $B(z_i^k, c_0 b^k /3) \subset Q_i^k \subset B(z_i^k, 2C_0 b^k)=:B(Q_i^k)$ for every $k\in \N$,
		\vspace{0.1cm}
		\item [(iv)] if $l\geq k$ and $Q_j^l\subset Q_i^k$, then $B(Q_j^l)\subset B(Q_i^k)$.
	\end{itemize}
	
	For non-negative $k\in \Z$, we call the sets $Q_i^k$ from the construction of Theorem~\ref{thm:Dydadic} ($b$-)\textit{dyadic cubes} of level $k$ of $X$.
	
\end{theoremA}

Fix $b$, $c_0$ and $C_0$ as in Theorem~\ref{thm:Dydadic}. Moreover, for every non-negative $k\in \Z$ we fix a collection of points $\{ z_i^k \}_{i\in I_k}$ and the corresponding collection of $b$-dyadic cubes $Q^k_i$. To see why such a collection of points exists, consider the covering $\{ B(z,c_0 b^k): z\in X \}$ of $X$ and apply the $5B$-covering lemma. By separability of $X$ and by choosing $c_0$ and $C_0$ so that $5c_0 b^k<C_0 b^k$, the existence of centers $\{ z_i^k \}_{i\in I_k}$ is ensured. Given a doubling metric space $X$, we fix such a system of dyadic cubes for the rest of the paper. We also denote by $\mathcal{D}_k$ the collection of all cubes in the fixed system which are of level $k$.

\begin{remark}\label{rem: k+1 level inside k}
    Note that given any $k$-level cube $Q$, there are at most a uniform number of cubes of level $k+1$ contained in $Q$, say $N_d\in \N$. This number $N_d$ solely depends on the doubling constant $C_d$ of $X$ and the constants $b$, $c_0$, $C_0$ of the dyadic system. Indeed, by Theorem~\ref{thm:Dydadic} (iii), for every cube of level $k+1$ inside $Q$, there is a ball of radius $3^{-1}c_0b^{k+1}$ inside $B(Q)$, which is of radius $2C_0 b^k$. Hence, by an application of the doubling property of $X$ we have at most
    $$
    N_d=C_d^{\left(\frac{2C_0b^k}{3^{-1}c_0b^{k+1}}\right) \log_2 C_d}=C_d^{\left(\frac{2C_0}{3^{-1}c_0b}\right) \log_2 C_d}
    $$such balls in $B(Q)$, which is the same bound on the number of $k+1$-level cubes inside $Q$.
\end{remark}

We show that the intermediate dimension can be expressed using the dyadic cube systems from Theorem~\ref{thm:Dydadic}. Note that a similar result in $\R^n$ was recently proved in \cite{FraTyson_intermed_dim} for the usual Euclidean dyadic cube system, although both the statement and the proof  differ from those in the metric setting that are detailed below.
\begin{proposition}\label{prop: dimdyadic}
Let $E\subset X$ be a non-empty subset of a $c_u$-uniformly perfect $C_d$-doubling metric space $X$. For $\theta\in (0,1)$, the $\theta$-intermediate dimension of $E$ is the infimum of the set $A_\theta$ consisting of exponents $s>0$ for which for all $\varepsilon > 0$ there exists $\delta_\varepsilon > 0$ such that for all $\delta\in (0,\delta_\varepsilon)$ there is a cover $\{Q_i\}_{i\in I_\delta}$ of $E$ by dyadic cubes $Q_i$ of level $k_i$ with
	\begin{equation}\label{eq: choice of levels dyadic spectrum}
	\frac{3}{c_u c_0} \, \delta^{1/\theta} \leq b^{k_i} \leq \frac{1}{4C_0} \delta,
		\end{equation}
	for all $i\in I_\delta$, and
	\[
	\sum_i |Q_i|^s < \varepsilon.
	\]
\end{proposition}

	\begin{proof}
	    
	Note that if $Q_i$ is a cube of level $k_i$, then by Theorem~A (iii) and uniform perfectness of $X$, we have
	\begin{equation*}\label{eq:thmA}
		\frac{c_u c_0}{3} b^{k_i} \leq |Q_i| \leq 4C_0 b^{k_i}.
	\end{equation*}
	So \eqref{eq: choice of levels dyadic spectrum} and the above ensure that 
		$$\delta^{1/\theta} \leq |Q_i| \leq \delta,$$ for all $i\in I_\delta$. 
    As a result, trivially $\dim_\theta E \leq \inf A_\theta$.
	
	Fix $s > \dim_\theta E$. We will show $s \geq \inf A_{\phi}$ for values $\phi<\theta$ as close to $\theta$ as desired. Then, by letting $s \to \dim_\theta E$, and using $\dim_{\phi} E \leq \inf A_{\phi}$, taking $\phi \to \theta$ and using the intermediate dimension's continuity yields the needed characterization.
        
        For $\varepsilon > 0$, there is $ \delta_\varepsilon = \delta_\varepsilon(s) > 0$ such that for all $ \delta < \delta_\varepsilon$
        there is a cover $\{U_i\}$ of $E$ with
        \begin{equation}\label{eq: Ui for dimθ}
        \delta^{1/\theta}\leq |U_i| \leq \delta
        \end{equation}
        and
        \begin{equation}\label{eq: sum Ui for dimθ}
        \sum |U_i|^s < \varepsilon.
        \end{equation}
        
        Let $k_i$ be the unique integer such that 
        \begin{equation}\label{eq: Ui compared bki}
        4C_0  b^{k_i} \leq |U_i| < 4C_0  b^{k_i-1}.
        \end{equation}
        
        \noindent
        \textit{Claim:} For every  $i\in I_\delta$, $U_i$ can be covered by at most
        \[
        N = C_d \Big( \frac{12C_0 b^{-1} + 24C_0}{c_0} \Big)^{\log_2 C_d}
        \]
        cubes of level $k_i$.
        
        \begin{claimproof}[Proof of Claim]
            For the rest of the proof of this Claim, we fix $i\in I_\delta$, and set $k=k_i$, $U=U_i$ to ease the notation.
        
        Set
        \[
        J^U_k = \{ j \in J_k : Q_j^k \cap U \neq \emptyset \}
        \] to be the set of indices of $k$-level cubes that intersect the set $U$.
        Moreover, set
        \[
        Q_U = \bigcup_{j \in J^U_k} Q_j^k, 
        \]
        to be the largest collection of $k$-level cubes that intersects and covers $U$.
        
        Let $x_0 \in Q_U$. Suppose $x_0 \in Q_U \setminus U$. Then there is $j_0 \in J^U_k$ such that 
        \[
        x_0 \in Q_{j_0}^k, \quad\text{and} \quad x_0' \in Q_{j_0}^k \cap U.
        \]
        Let $x \in Q_U$. There are two cases to consider:
        
        \textit{Case 1:} If $x \in U$, then
        \begin{align*}
        d(x_0,x) 
        &\leq d(x_0,x_0') + d(x_0',x) \\
        &\leq |Q_{j_0}^k| + |U| \\
        &\leq 4C_0 b^k + (4C_0)b^{k-1} \\
        &= (4C_0 b^{-1} + 4C_0)b^k.
        \end{align*}
        
        \textit{Case 2:} If $x \in Q_U \setminus U$, then there is $ j_x \in J^U_k$ such that 
        \[
        x \in Q_{j_x}^k, \quad x' \in Q_{j_x}^k \cap U.
        \]
        
        Then,
        \begin{align*}
            d(x_0,x) 
            &\leq d(x_0,x_0') + d(x_0',x') + d(x',x) \\
            &\leq |Q_{j_0}^k| + |U| + |Q_{j_x}^k| \\
            &\leq 4C_0 b^k + |U| + 4C_0 b^k \\
            &\leq (4C_0 b^{-1} + 8C_0)b^k.
        \end{align*}
        
        All the above was under the assumption that the arbitrary point $x_0\in Q_U$ is not contained in $U$. Similarly, if $x_0 \in U$, it can be shown that
        \[
        d(x_0,x) \leq (4C_0 b^{-1} + 4C_0)b^k, 
        \] for all $x \in Q_U$.
        As a result,
        \[
        d(x_0,x) \leq (4C_0 b^{-1} + 8C_0)b^k,
        \]for all $x_0,x \in Q_U$,
        implying that 
        $$|Q_U| \leq (4C_0 b^{-1} + 8C_0)b^k.$$
        
        Thus, there is some $y \in Q_U$ such that  $Q_U \subset B(y,(4C_0 b^{-1} + 8C_0)b^k)$. For simplicity, we denote this ball by $B$.
        Note that every cube in $Q_U$, by Theorem~A (iii), contains a ball of radius $\tfrac{c_0b^k}{3}$, and all these balls lie in $B$ and are disjoint. Hence, by the doubling condition of $X$, there are at most 
        \[N= C_d^{\frac{(4C_0 b^{-1} + 8C_0)b^k}{c_0b^k/3}\log_2 C_d}
        \] such balls of radius $\tfrac{c_0b^k}{3}$ inside $B$, which implies that there are at most this many $k$-level cubes to cover $U$.
        
        \end{claimproof} 
        
        For every set $U_i$, there is a collection of $k_i$-level cubes $\{Q_{j}^{k_i}\}_{j\in J_i}$ that cover $U_i$. We use the cubes $\mathcal{Q_\delta}:=\{Q_{j}^{k_i}\}_{\substack{j \in J_i \\ i \in I_\delta}}$ to cover $E$.
        
        Note that, by \eqref{eq: Ui compared bki},  we have for all $i\in I_\delta$ that
        \begin{equation}\label{eq: cube admiss}
        \frac{b \,\delta^{1/\theta}}{4C_0}\leq b \,\frac{|U_i|}{4C_0}\leq b^{k_i} \leq \frac{|U_i|}{4C_0} \leq \frac{\delta}{4C_0},
        \end{equation}
        
        Because of \eqref{eq: cube admiss}, the cover $\mathcal{Q_\delta}$ might not be admissible for the conditions of $A_\theta$. Thus, we need to choose a $\phi$ for which $b^{k_i}$ satisfies the appropriate lower bound for $\mathcal{Q_\delta}$ to be an admissible collection for the conditions in $A_{\phi}$.
        Set
        \begin{equation}
        \phi := \frac{\theta\log \delta_\varepsilon}{\log \delta_\varepsilon + \theta \log \frac{b\, c_u\, c_0}{12C_0}},
        \label{eq: new θ ie phi}
        \end{equation}
        and note that $\phi < \theta$.
        With this choice we have 
        $$
        \delta_\eps^{\frac1\phi - \frac1\theta} = \frac{b c_uc_0}{12C_0},
        $$
        which ensures that for all $\delta \leq \delta_\eps$ we have
        \[
        \frac{3}{c_uc_0}\delta^{1/\phi} \leq \frac{b}{4C_0}\delta^{1/\theta}.
        \]
        Thus, \eqref{eq: cube admiss}, \eqref{eq: new θ ie phi} imply 
        $$
        \frac{3}{c_uc_0}\delta^{1/\phi}\leq b^{k_i}  \leq \frac{\delta}{4C_0},
        $$
         which makes the collection $\{Q_i^{k_i}\}$ admissible for $\inf A_{\phi}$.
        
        Also, due to the Claim,
        \[
        \sum_{i,j} |Q_{j}^{k_i}|^s \leq N \; \sum_i |U_i|^s < N \; \varepsilon := \varepsilon'.
        \]
        
        As a result, because $\phi \to \theta$ as $\delta_\varepsilon \to 0$, we have shown that for every tiny $\lambda > 0$, for every $\varepsilon' > 0$, there is $\delta_{\varepsilon'}'$ small enough (as dictated by \eqref{eq: new θ ie phi} and the value $\delta_{N^{-1}\varepsilon'}$) such that for all $\delta \in (0,\delta'_{\varepsilon'})$, there is a cover $\{Q_i\}_{i\in I_\delta}$ of  $E$ by cubes of level $k_i$ with  
        $$\frac{3}{c_uc_0}\delta^{1/(\theta+\lambda)} \leq b^{k_i} \leq \frac{1}{4C_0}\delta,
        $$
        so that
        \[
        \sum_i |Q_i|^s < \varepsilon.
        \]
        
        This implies
        \[
        \dim_{\theta+\lambda} E\leq \inf A_{\theta+\lambda} \leq s.
        \]
        
        Since $s > \dim_\theta E$ was arbitrary, we let $s\rightarrow \dim_\theta E$, and then letting $\lambda \to 0$, by continuity of intermediate dimensions, completes the proof.
    \end{proof} 

\begin{remark}
The construction of dyadic cubes in \cite{Hyt:dyadic} was actually given for quasimetric spaces. As a result, Proposition~\ref{prop: dimdyadic} is also true if $X$ is a quasimetric doubling space. The proof is almost identical, with the only difference being the dependence of a few of the constants on the quasimetric constant of the space.
\end{remark}

\subsection{Mappings between metric spaces}\label{subsec:BackSobolev}
Let $(X,d_X)$ and $(Y,d_Y)$ be two metric spaces.
Given $\alpha\in (0,\infty)$, a mapping $f:X\rightarrow Y$ and a set $B\subset X$, we  define the \textit{$\alpha$-H\"older coefficient} of $f$ on $B$ as
$$
|f|_{\alpha, B}:= \sup\left\{ \frac{d_Y(f(x),f(y))}{[d_X(x,y)]^\alpha}: \, x, y \in B \,\, \text{distinct} \right\}.
$$ If $|f|_{\alpha, B}<\infty$ then we say that $f$ is \textit{$\alpha$-H\"older continuous} in $B$. 

Given an at most countable index set $I$, we denote by $\ell^p(I)$, $p\in (0,\infty)$, the space of real-valued sequences $\{c_i\}_{i\in I}$ with finite $p$-norm $(\sum_{i\in I} c_i^p)^{1/p}<\infty$. We call $\sum_{i\in I} c_i^p$ the \textit{$p$-sum} of the sequence $\{c_i\}_{i\in \N}$.

For the rest of the paper, all index sets are assumed to be at most countable. We now recall the class of compactly H\"older mappings.

\begin{definition}\label{def: CH maps}
    Let $f:X\rightarrow Y$ be a mapping between two arbitrary metric spaces. For $p, \alpha\in(0,\infty)$, we say $f$ is \textit{$(p,\alpha)$-compactly H\"older}, and write $f\in CH^{p,\alpha}(X:Y)$, if for any compact set $E\subset X$ and any $\eps\in (0,1)$ there are $r_E>0$ and $C_E>0$  satisfying the following: \\
if $\{B_i\}_{i\in I}$ is a collection of balls $B_i:=B(x_i,r_i)$ with $x_i\in X$, $r_i<r_E$ that covers $E$ and $B(x_i,\eps r_i)\cap B(x_j,\eps r_j)=\emptyset$ for all distinct $i, j\in I$, then the $p$-sum of the H\"older coefficients of $f$ on $B_i$ is at most $C_E$, i.e.,
\begin{equation}\label{eq: CH-def-inequality}
    \sum\limits_{i\in I} |f|_{\alpha, B_i}^p\leq C_E.
\end{equation}
\end{definition}

Here we follow the convention that if $\{B(x_i,r_i)\}_{i\in I}$ covers $E$, it is implied that $B(x_i,r_i)\cap E\neq\emptyset$ for all $i$, but not all $x_i$  necessarily lie in $E$. Note that  applying the definition on singleton sets yields that compactly H\"older mappings are H\"older continuous on compact sets. Moreover, in the setting of Definition~\ref{def: CH maps} it is actually implied by \eqref{eq: CH-def-inequality} that there are $C_i>0$ such that
\begin{equation}\label{eq: CH-inequality}
    |f(B(x_i,r_i)|\leq C_i |B(x_i,r_i)|^\alpha
\end{equation} with $\sum_{i\in I}C_i^p\leq C_E$. This inequality 
provides insight on the relation with the Euclidean setting. More specifically, it hints at how the motivation for Definition~\ref{def: CH maps} comes from continuous super-critical Sobolev maps between Euclidean spaces, i.e. continuous maps in $W^{1,p}(\Omega;\R^n)$ with $\Omega\subset \R^n$ and $p>n$. See also \cite{Chron_comp_holder_Minkowski}.
\begin{remark}\label{rem: old def}
    In \cite{Chron_comp_holder_Minkowski} the class of compactly H\"older maps was initially defined with $\{B_i\}$ being a cover where all balls are of the same radius $r$ (see \cite[Definition 2.3]{Chron_comp_holder_Minkowski}). While that condition is enough to determine the distortion of the Minkowski dimension, all types of Sobolev maps considered in this manuscript and in \cite{Chron_comp_holder_Minkowski} satisfy the stronger condition with potentially distinct radii $r_i$. Thus, we decided to update the definition to require this stronger property, especially since the motivation is to appropriately generalize these mapping classes to a solely metric setting (with no measure). It should be noted that all results in \cite{Chron_comp_holder_Minkowski} stated for the ``older'' and wider class of compactly-H\"older mappings also hold for the notion in Definition~\ref{def: CH maps}.
\end{remark}

We now turn to discussing fractionally smooth mappings in the metric setting, which requires a measure. A triplet $(X,d_X,\mu)$ is called a metric measure space  if $(X,d_X)$ is a metric space and $\mu$ is a Borel measure on $X$ that assigns a strictly positive and finite value on all balls in $X$. Thus, throughout the paper all measures are considered to have the aforementioned properties, even if not stated explicitly. Note that every metric measure space is necessarily separable (see \cite{Gorka21}). For $p\in (0,\infty]$ we denote the space of \textit{$p$-integrable} real-valued functions defined on $X$ by $L^p(X,\mu)$, or simply by $L^p(X)$ if the measure follows from the context, and by $L^p_\loc(X)$ the space of locally $p$-integrable real-valued functions defined on $X$. Moreover, for a ball $B\subset X$ and $u\in L^1(B)$ we denote by $u_B$ the average of $u$ over $B$, i.e., $u_B:= \Barint_{B}u\, d\mu = \mu(B)^{-1} \int_{B} u\, d\mu$.

Let $(X,d_X,\mu)$ be a metric measure space, $(Y,d_Y)$ be a metric space equipped with the Borel sigma-algebra, and $s\in(0,\infty)$.
Following \cite{HajSob,hajlasz,Y03},
a measurable function  $g:X\to[0,\infty]$  is called
an \textit{$s$-gradient} of a measurable function $u\colon  X\rightarrow Y$ if
there exists a set $E\subset X$ with $\mu(E)=0$ such that
\begin{equation}
\label{fracHajlasz}
d_Y(u(x),u(y))\leq [d_X(x,y)]^s\left[g(x)+g(y)\right],
\end{equation}
for every $x,y\in X\setminus E$.
The collection of all the $s$-gradients of $u$ is denoted by $D^s(u)$.
Given $p\in(0,\infty)$, the \textit{(homogeneous) fractional Haj\l asz--Sobolev space}
$\dot{M}^{s,p}(X:Y)$ is defined as the collection of
all the measurable functions $u\colon  X\rightarrow Y$ such that
\begin{equation*}
\Vert u\Vert_{\dot{M}^{s,p}(X:Y)}:=
\inf_{g\in D^s(u)}\Vert g\Vert_{L^p(X)}<\infty.
\end{equation*}
Here and thereafter, we make the agreement that for semi-norms $\inf\emptyset:=\infty$.

In order to define the Haj\l asz--Triebel--Lizorkin and Haj\l asz--Besov  spaces,
we need a suitable notion for the gradient.
Following \cite{KYZ11}, a sequence $\{g_k\}_{k\in\mathbb{Z}}$ of measurable
functions $g_k:X\to[0,\infty]$ is called a \textit{fractional $s$-gradient} of
a measurable function $u\colon X\rightarrow Y$ if there exists
a set $E\subset X$ with $\mu(E)=0$ such that
\begin{equation}\label{Hajlasz}
d_Y(u(x),u(y))\leq [d_X(x,y)]^s \left[g_k(x)+g_k(y)\right]
\end{equation}
for any $k\in\mathbb{Z}$ and  $x,y\in X\setminus E$ satisfying $2^{-k-1}\leq d_X(x,y)<2^{-k}.$
Let $\mathbb{D}^s(u)$ denote the set of all the fractional $s$-gradients of $u$.
Given $p\in(0,\infty)$, $q\in(0,\infty]$, and a sequence  $\vec{g}:=\{g_k\}_{k\in\mathbb{Z}}$
of measurable functions $g_k:X\to[0,\infty]$, define
\begin{equation*}
\Vert \vec{g}\Vert_{L^p(X,\ell^q)}:=
\lf\Vert\, \Vert \{g_k\}_{k\in\mathbb{Z}}\Vert_{\ell^q} \r\Vert_{L^p(X)}
\end{equation*}
and
\begin{equation*}
\Vert \vec{g}\Vert_{\ell^q(L^p(X))}:=\lf\Vert \lf\{\Vert g_k\Vert_{L^p(X)}\r\}_{k\in\mathbb{Z}} \r\Vert_{\ell^q},
\end{equation*}
where
\begin{equation*}
\Vert \{g_k\}_{k\in\mathbb{Z}}\Vert_{\ell^q}:=
\begin{cases}
 \displaystyle\left(\sum_{k\in\mathbb{Z}}\vert g_k\vert^q\right)^{1/q}& ~\text{if}~q\in(0,\infty),\\
 \displaystyle \sup_{k\in\mathbb{Z}}\vert g_k\vert& ~\text{if}~q=\infty.
\end{cases}
\end{equation*}
Then the
\textit{(homogeneous) Haj\l asz--Triebel--Lizorkin space}
$\dot{M}^s_{p,q}(X:Y)$ is defined as the
collection of all the measurable mappings  $u\colon  X\rightarrow Y$ such that the semi-norm
\begin{equation*}
\Vert u\Vert_{\dot{M}^s_{p,q}(X:Y)}:=\inf_{\vec{g}\in
\mathbb{D}^s(u)}\Vert\vec{g}\Vert_{L^p(X,\ell^q)}<\infty.
\end{equation*}
The \textit{(homogeneous) Haj\l asz--Besov space} $\dot{N}^s_{p,q}(X:Y)$
is defined as the collection of all the measurable mappings
$u\colon  X\rightarrow Y$ such that the semi-norm
\begin{equation*}
\Vert u\Vert_{\dot{N}^s_{p,q}(X:Y)}:=
\inf_{\vec{g}\in\mathbb{D}^s(u)}\Vert\vec{g}\Vert_{\ell^q(L^p(X))}<\infty.
\end{equation*}

A few comments are in order. We use the term `semi-norms' for $\Vert u\Vert_{\dot{M}^{s,p}(X:Y)}$, $\Vert \cdot\Vert_{\dot{M}^s_{p,q}(X:Y)}$, and $\Vert u\Vert_{\dot{N}^s_{p,q}(X:Y)}$ even though  the triangle inequality holds only when $p,q\geq1$.
A genuine `norm' can be obtained by passing to the quotient space modulo constant functions.
Since altering functions on sets of measure zero do not affect their
membership in $\dot{M}^{s,p}(X:Y)$, $\dot{M}^s_{p,q}(X:Y)$, or $\dot{N}^s_{p,q}(X:Y)$, it is standard to regard these spaces as consisting of equivalence classes of functions. We adopt this convention here, but we choose to omit the details.

When $Y=\mathbb{R}$, we make the following abbreviations:
$\dot{M}^{s,p}(X):=\dot{M}^{s,p}(X:\mathbb{R})$
$\dot{M}^s_{p,q}(X):=\dot{M}^s_{p,q}(X:\mathbb{R})$, and $\dot{N}^s_{p,q}(X):=\dot{N}^s_{p,q}(X:\mathbb{R})$. 
It was shown in \cite{KYZ11} that $\dot{M}^s_{p,q}(\mathbb{R}^n)$ coincides with the classical Triebel--Lizorkin space $\dot{F}^s_{p,q}(\mathbb{R}^n)$ for any $s\in(0,1)$, $p\in(\frac{n}{n+s},\infty)$, and $q\in(\frac{n}{n+s},\infty]$,  and $\dot{N}^s_{p,q}(\mathbb{R}^n)$ coincides with the classical Besov space $\dot{B}^s_{p,q}(\mathbb{R}^n)$ for any $s\in(0,1)$, $p\in(\frac{n}{n+s},\infty)$, and $q\in(0,\infty]$. In particular, the Haj\l{}asz--Triebel--Lizorkin and Haj\l{}asz--Besov spaces on $\mathbb{R}^n$ contain the classical fractional Sobolev spaces as special cases (see \cite[Chapter~2]{GrafModMucken}).
When $s=1$, we have that $\dot{M}^1_{p,\infty}(\mathbb{R}^n)=\dot{M}^{1,p}(\mathbb{R}^n)$ coincides with the classical Sobolev space $\dot{W}^{1,p}(\mathbb{R}^n)$ for any $p\in(1,\infty)$; see Lemma~\ref{sobequal} and \cite{hajlasz}. We refer the reader to
\cite{HajSob,hajlasz,Y03,KYZ11,HKT07,agh20,hhhpl21,AWYY21,AGS25} for more information on Sobolev, Triebel--Lizorkin, and Besov spaces on metric measure spaces.

The following notions for measures are also typically assumed in this setting. We say a metric measure space $(X,d,\mu)$ is \textit{locally $Q$-homogeneous}, for some $Q>0$, if for all compact $K\subset X$ there are constants $\tilde{R}_{\text{hom}}(K)>0$, $\tilde{C}_{\text{hom}}(K)\geq 1$ such that
\begin{equation}\label{Qhom}
\frac{\mu(B(x,r_2))}{\mu(B(x,r_1))}\leq \tilde{C}_{\text{hom}}(K) \left(\frac{r_2}{r_1}\right)^Q,
\end{equation}
for all $x\in K$ and scales $0<r_1<r_2<\tilde{R}_{\text{hom}}(K)$. Note that local $Q$-homogeneity implies that \((X, d, \mu)\) is \textit{locally doubling} (see\cite{BaloghAGMS}), i.e., for every compact subset $K \subseteq X$, there exists a radius $R > 0$ and a constant $C \geq 1$ such that
$$
\mu(B(x, 2r)) \leq C \, \mu(B(x, r))
$$
whenever \(B(x, r)\) is a ball centered at a point in \(K\) with \(r \leq \tilde{R}_{\text{hom}}(K)\).
We say a metric space $(X,d)$ is \textit{locally $Q$-homogeneous} if there is a measure $\mu$ on $X$ such that $(X,d,\mu)$ is locally $Q$-homogeneous. One particular property due to local homogeneity that we  need is the lower bound on the measure of a ball by its radius to a power. More specifically, for $R_{\text{hom}}(K)=\tilde{R}_{\text{hom}}(K)/3$ and a potentially larger $C_{\text{hom}}(K)$, it can be shown that 
\begin{equation}\label{eq:lower-AR}
    \frac{r^Q}{C_{\text{hom}}(K)} \leq \mu (B(x,r)),
\end{equation} for all $x\in K$ and $r\in (0,R_{\text{hom}}(K))$.

It can be necessary at times to also have a similar upper bound on the measure, and especially in the study of quasisymmetric mappings. We say that $(X,d,\mu)$ is \textit{$Q$-Ahlfors regular} for some $Q>0$ if there is a constant $C_A>0$ such that for all $x\in X$ and all $r\in(0,|X|)$ we have
    $$
    \frac{1}{C_A}r^Q \leq \mu (B(x,r)) \leq C_A r^Q.
    $$ 
    We say a metric space $(X,d)$ is \textit{$Q$-Ahlfors regular} if there is a measure $\mu$ on $X$ such that $(X,d,\mu)$  is $Q$-Ahlfors regular. Note that $Q$-regularity of a measure implies the $Q$-homogeneous property.

\section{Intermediate dimensions under compactly H\"older mappings}\label{sec:DimDistProofs}

Suppose $(Y,d_Y)$ is a $c_u'$-uniformly perfect metric space and $(X,d_X)$ is a $c_{u}$-uniformly perfect, $C_d$-doubling 
metric space with a fixed system of dyadic cubes as in Section~\ref{sec:background}. We further assume that both $X$ and $Y$ have more than one point each, as otherwise all results trivially hold. 
Before we proceed, we need the following elementary property of uniformly perfect spaces.
		\begin{lemma}\label{lem: unif perfect L(Ui)} 
        If $Y$ is $c_u'$-uniformly perfect, then for every $F\subset Y$ and every $M>0$ with $|F|<M<M c_u'^{-1}<|Y|$,
		there is a set $L(F)$ containing $F$ with
		\[
		M \leq |L(F)| \leq 2c_u'^{-1} M.
		\]
        \end{lemma}
		
		\begin{proof}
		    Let $y\in F$, then $B_Y(y, M c_u'^{-1})$ contains $F$, and
		\[
		|B_Y(y,M c_u'^{-1})| \geq c_u' c_u'^{-1} M = M
		\]
		and
		\[
		|B_Y(y,M c_u'^{-1})| \leq 2M c_u'^{-1}.
		\]
		Hence, $L(F)=B_Y(y, M c_u'^{-1})$ is the desired set.
		\end{proof}

Since all mappings considered henceforth are  defined on $X$ and into $Y$, we set $CH^{p,\alpha}=CH^{p,\alpha}(X:Y)$. Let $f\in CH^{p,\alpha}$ for $p\in (1,\infty)$, $\alpha>0$ and $E\subset X$ be a bounded set. 
We plan on using coverings of $E$ by dyadic cubes, provided by Proposition~\ref{prop: dimdyadic}, and determine  admissible coverings for $\dim_\theta f(E)$ based on those of $E$ and shrinking properties of $f$ controlled by inequality \eqref{eq: CH-inequality}. 
Before we delve into the proof, we give an  outline of the combinatorial part of our method.
The main idea is to fix an arbitrary $d>\dim_\theta E$, and for every $\varepsilon>0$ fix for every $\delta\in (0,\delta_\varepsilon)$ a dyadic covering $\mathcal{Q}_\delta$ of $E$ as in Proposition~\ref{prop: dimdyadic}. The images of the covering cubes form a cover of $f(E)$. After determining a $\delta_Y$ to use for $\dim_\theta f(E)$ (see \eqref{eq: δ in f(E)}), we need to modify the cubes in order for their images to yield a $\delta_Y^{1/\theta}$-admissible cover of $f(E)$. To do so, we build a combinatorial graph with cubes as vertices, in the following manner (see Figure~\ref{fig: Graph}): 
\begin{itemize}
    \item Let $\{m_1, \dots, m_{N_\delta}\}$ be the set of all levels of cubes in $\mathcal{Q}_\delta$, with $m_i<m_j$ for all $i<j$. List all $m_1$-level cubes in $\mathcal{Q}_\delta$ on the first row, all $m_2$-level cubes in $\mathcal{Q}_\delta$ $m_2-m_1$ rows lower, all $m_3$-level cubes in $\mathcal{Q}_\delta$ $m_3-m_2$ rows lower and so on, until the highest level $m_{N_\delta}$ for cubes in $\mathcal{Q}_\delta$.
    \item Color all cubes whose image under $f$ has diameter larger than $\delta_Y$ red, those with images of diameter less than $\delta_Y^{1/\theta}$ green, and the remaining blue.
    \item Subdivide every red cube, based on Theorem~\ref{thm:Dydadic}, into next level sub-cubes, and draw in the next row from their ancestor those that intersect $E$, and connect them with  edges to their ancestor. These descendant cubes could be red, blue, or green, depending on the behavior of $f$.
    \item Iterate the previous step, until we reach the first level $m_{K_\delta}\geq m_{N_\delta}$ with no red cubes. This is guaranteed by uniform continuity of $f$.
\end{itemize}

		\begin{figure}
			\centering
            \includegraphics{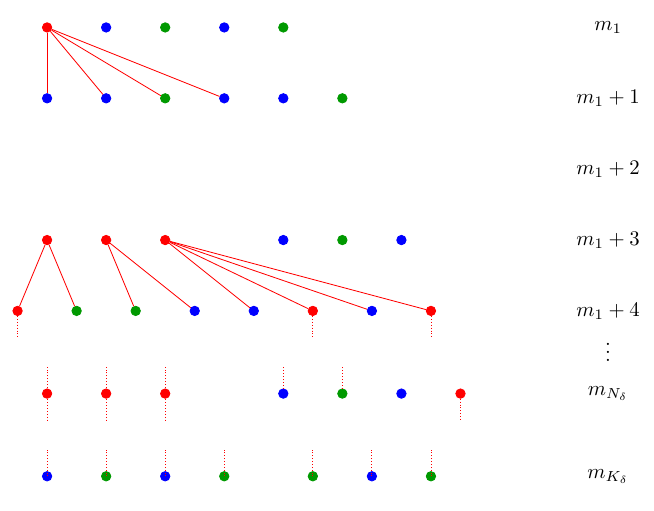}
			\caption{The graph $\mathcal{G}(\mathcal{Q}_\delta)$ above includes a few noteworthy characteristics of $\mathcal{Q}_\delta$ and $f$. It is depicted that $\mathcal{Q}_\delta$ has no cubes of level $m_1+2$, and while it has no cubes of level $m_1+4$ either, we need to subdivide red cubes to and past that level for the induced desired covering of $f(E)$. Moreover, there are red cubes of $\mathcal{Q}_\delta$ even at the very last level, which means that $f$ is capable of increasing the diameter of even these small cubes, creating the need to subdivide beyond the $m_{N_\delta}$ level.}\label{fig: Graph}
		\end{figure}
        Based on the above graph, the strategy then is to use a covering of $f(E)$ consisting of images of all blue cubes, and ``enlarged'' images of green cubes (using Lemma~\ref{lem: unif perfect L(Ui)}). This by construction is a $\delta_Y^{1/\theta}$-admissible covering of $f(E)$, with the least cardinality possible among coverings based on $\mathcal{Q}_\delta$, and with a combinatorial representation that facilitates estimates on the sum of respective diameters to an appropriate exponent, achieving the desired upper bound on $\dim_\theta f(E)$. With this strategy in mind, we are ready to delve into the proof of Theorem~\ref{thm:main_Holder}.

\begin{proof}[Proof of Theorem~\ref{thm:main_Holder}]
		Fix $\theta\in (0,1)$. We focus on the case $\dim_\theta E=d_E(\theta)<pd_E(\theta)(\alpha p+d_E(\theta))^{-1}$, and the other case can be treated similarly. Suppose $d_E := d_E(\theta) < d < d' < p-\alpha p$, and set 
        $D=\tfrac{p{d'}}{\alpha p + d'}$. The choice $d<d' < p-\alpha p$ guarantees that  $d \neq D$, but the proof is similar in the case $d_E(\theta)\geq pd_E(\theta)(\alpha p+d_E(\theta))^{-1}$, 
		by choosing arbitrary $D>d'>d>d_E$.

		Due to $d>\dim_\theta E$ and Proposition~\ref{prop: dimdyadic}, for $\varepsilon>0$ there is $\delta_\varepsilon\in (0,1)$, such that for all
		$ \delta\in(0,\delta_\varepsilon)$ there is a covering $ \mathcal{Q_\delta}:=\{Q_i^{k_i}\}_{i\in I_\delta}$ of  $E$ by dyadic cubes
		such that 
		\begin{equation}\label{eq:1}
			\frac{3}{c_uc_0}\,\delta^{1/\theta} \leq b^{k_i} \leq \frac{1}{4C_0}\,\delta,
		\end{equation}		
		for all $i\in I_\delta$, and
		\begin{equation}\label{eq:2}
			\sum_{i\in I_\delta} |Q_i^{k_i}|^d < \varepsilon.
		\end{equation}
		Without loss of generality, we may assume that $\delta_\varepsilon$ is small enough for certain properties to apply. Namely, we assume that 
        \begin{itemize}
            \item $\delta_\varepsilon < \varepsilon < 1$, in order to replace $\delta$ by $\varepsilon$ when all that matters is for a quantity to be small,
            \item $\delta_\varepsilon<r_E$, to ensure that \eqref{eq: CH-inequality} can be applied to the corresponding balls $B(Q_i^{k_i})$ from Theorem~\ref{thm:Dydadic} (iii) for the cubes in $\mathcal{Q}_\delta$, and all their sub-cubes,
            \item $\delta_\varepsilon<2^{-1}\min\{|X|, (C_E^{-1}|Y|)^{1/\alpha}\}$, to ensure that the uniformly perfect properties of $X$ and $Y$ can be applied to balls $B(Q_i^{k_i})$ in $X$, and balls of radius $|f(Q_i^{k_i})|$ in $Y$, respectively.
            \item $\delta_\varepsilon<(2^{-1}|Y|/c_u')^{D/d}$, to ensure that the uniformly perfect property of $Y$ can be applied to balls of radius $\delta^{d/(D\theta)}/c_u'$ in $Y$ (i.e., Lemma~\ref{lem: unif perfect L(Ui)} applies for sets $F$ with $|F|<M=\de^{d/(D\theta)}$).
        \end{itemize} Without loss of generality also assume $k_1\leq k_2\leq \;s \leq k_{N_\delta}$. Moreover, if $Q_{i'}^{k_{i'}}\subset Q_i^{k_i}$ for $i'\neq i\in I_\delta$,
		then we can reduce $\{Q_i^{k_i}\}_{i\in I_\delta}$ to $\{Q_i^{k_i}\}_{i\in I_\delta\setminus\{i'\}}$, and \eqref{eq:1}, \eqref{eq:2}
		would still be true for $i\in I_\delta\setminus\{i'\}$, since $Q_{i'}^{k_{i'}}$ would be redundant for the covering. As a result, we may assume that if $Q_i^{k_i}$ is a cube in the cover
		$\mathcal{Q_\delta}$, then no sub-cube of $Q_i^{k_i}$
		is contained in $\mathcal{Q_\delta}$. Note that $k_i$’s are not necessarily
		pairwise distinct, since, for instance, there could be
		two level $m=k_1=k_2$ cubes in $\mathcal{Q_\delta}$. We  consider a re-labeling of
		the cubes in $\mathcal{Q_\delta}$ to account for that, namely,
		set 
        $$
        I_\delta' := \{k_i: i\in I_\delta\},
        $$ which can be represented
		as $I_\delta' = \{m_1,m_2,\dots,m_{N_\delta}\}$, with positive integers $m_i$, with $m_i<m_j$ for all $i<j$.
		This notation helps by clarifying exactly which level we address in following arguments. We also set
		\begin{equation}\label{eq: δ in f(E)}
			\delta_Y := \delta^{d/D}.
		\end{equation}

		We say a cube $Q$ that intersects $E$ and is contained in some $Q_i^{k_i}\in \mathcal{Q}_\delta$, for some $i\in I_\delta$,  is
		\textit{blue} if $|f(Q)|\in[\delta_Y^{1/\theta},\delta_Y)$, \textit{green} if $|f(Q)|< \delta_Y^{1/\theta}$,
		and \textit{red} otherwise. Notice that the terminology
		applies not only to cubes in $\mathcal{Q_\delta}$, but also
		to their sub-cubes. In addition, it applies \textit{only} to the aforementioned types of cubes, due to our initial reduction on $\mathcal{Q}_\de$ and by Theorem~\ref{thm:Dydadic} (i). By uniform continuity of $f$,
		it is guaranteed that past a certain level, all
		sub-cubes of cubes in $\mathcal{Q_\delta}$ will be  blue
		or green. The strategy is to sub-divide every
		red cube in $\mathcal{Q_\delta}$ and all their red sub-cubes the least number of times
		necessary, in order for the resulting cubes that intersect $E$ to all be
		blue or green. Then, apply  Lemma~\ref{lem: unif perfect L(Ui)} to the images of all green cubes under $f$, in order to ``enlarge'' them, and pick a covering of $f(E)$ consisting of these enlarged sets and the images of all blue
		cubes under $f$. By construction, this covering is a $\delta_Y^{1/\theta}$-admissible covering for
		$\dim_\theta f(E)$. Then, it would be enough to show that
		the sum of their diameters raised to $D$ is small enough.
		
		Suppose $m_{K_\delta} \geq m_{N_\delta}$ is the smallest integer
		such that all cubes of level $m_{K_\delta}$ which intersect $E$ and
		are contained in $\bigcup_{i\in I_\delta} Q_i^{k_i}$, are either green or blue.
		Note that $m_{K_\delta}$ could be strictly larger than
		$m_{N_\delta}$, because we can have red $m_{N_\delta}$-level
		cubes in $\mathcal{Q_\delta}$, or red cubes in $\mathcal{Q_\delta}$ might need
		to be subdivided past the level $m_{N_\delta}$ to give only blue
		and green descendants. 
        Based on this subdivision of red cubes in $\mathcal{Q_\delta}$ and their red sub-cubes, solely dependent on the way $f$ distorts diameters,
		we build a directed graph $\mathcal{G}=\mathcal{G}(\mathcal{Q_\delta})$. For an integer $m\in [m_1,m_{K_\delta}]$, set $V_m$ to be the collection of all $m$-level cubes intersecting $E$ and contained in $\cup_{i\in I_\delta}Q_i^{k_i}$ that are either red, or are contained in a red cube of level $m-1$. This excludes descendants of blue and green cubes from being included in $V_m$, because those are not subdivided. Then, the vertex set of $\mathcal{G}$ is defined to be the collection of cubes
		\[
		V= \mathcal{Q}_\delta\cup\bigcup_{m=m_1}^{m_{K_\delta}} V_m,
		\]
		and the edges of $\mathcal{G}$ are
        $$
        E = \bigcup_{m=m_1}^{m_{K_\delta}-1} \big\{(Q_t,Q_s):\; Q_t\in \mathcal{D}_{m},Q_s\in \mathcal{D}_{m+1},\; Q_s\subsetneq Q_t,\; \text{and } Q_t \text{ red}\big\}.
        $$
        What we use in the desired covering of $f(E)$ is essentially the images of blue vertices in $V$ and the ``enlarged'' images of green vertices in $V$ (after applying Lemma~\ref{lem: unif perfect L(Ui)}). For an integer $m\in [m_1, m_{K_\delta}]$, we  call the collection of all $m$-level cubes in $V$, i.e. the collection $\mathcal{Q}_\delta \cup V_m$, the $m$-th \textit{row} in $\mathcal{G}$. (See also Figure~\ref{fig: Graph}).
        
        Note that blue and green cubes in $\mathcal{Q_\delta}$ are part of no edge, because there is no need for them to be sub-divided, and they cannot have an ancestor also in $\mathcal{Q}_\delta$ by the respective reduction in the beginning of the proof.
        Let $V_R, V_B^{\text{orig}}, V_B', V_G$ denote the red cubes
		in $V$, blue cubes in $V\cap \mathcal{Q_\delta}$, blue cubes in $V$ that
		do not lie in $\mathcal{Q_\delta}$ (i.e. blue cubes with an ancestor in $V$),
		and green cubes in $V$, respectively.
		Notice that if we have a bound on $\text{card}V_R$, we can also bound $\text{card}V_B'$
		by Remark~\ref{rem: k+1 level inside k}:
		\begin{equation}\label{eq: blue desc}
			\text{card}V_B' \leq N_d\; \text{card}V_R,
		\end{equation}
		because every vertex in $V_B'$ is connected to
		a red ancestor vertex in $V$, and for each red $Q\in V$
		there are at most $N_d$ cubes $\tilde Q\in V$ with $(Q,\tilde Q)\in E$.
		
		Note, by \eqref{eq:1} and \eqref{eq:2}, that
		\[
		\text{card} \mathcal{Q_\delta} \leq c_{u,0} \,\varepsilon\,\delta^{-d/\theta},
		\]
		where $c_{u,0}=(\frac{c_uc_0}{3})^d$. This implies
		\begin{equation}\label{eq:5}
			\text{card} V_G \leq 2\max\Big\{c_{u,0}\,\varepsilon\,\delta^{-d/\theta},\; N_d\; \text{card}V_R\Big\},
		\end{equation}
        where $c_{u,0}\varepsilon\delta^{-d/\theta}$ bounds the number of green cubes in $\mathcal{Q}_\delta$, and $N_d\, \text{card} V_R$ bounds the number of green cubes in $V$ which are connected to some red ancestor vertex in $V$ by an edge, after another application of Remark~\ref{rem: k+1 level inside k}. Let
        $$
        \mathcal{U}=\mathcal{U}_\delta=\{ f(Q): Q\in V_B^{\text{orig}}\cup V_B'\}\cup \{ L(f(Q)): Q\in V_G\},
        $$
        where $L(f(Q))$ is a set as in Lemma~\ref{lem: unif perfect L(Ui)} applied to $f(Q)$  for $M=\delta_Y^{1/\theta}$. By definition of blue and green
		cubes and by Lemma~\ref{lem: unif perfect L(Ui)}, which is applicable due to assuming that $\delta_\varepsilon$ is
		small enough so that $2c_u'^{-1}\delta_Y^{1/\theta}\leq \delta_Y$ for $\delta\leq \de_\varepsilon$ (see last point after \eqref{eq:2}), the collection
		$\mathcal{U}$ is a $\delta_Y^{1/\theta}$-admissible cover of $f(E)$.
		Thus, we need to estimate the sum
		\[
		\mathcal{S}:=\sum_{U\in\mathcal{U}} |U|^{D} =
		\sum_{Q\in V_B^{\text{orig}}} |f(Q)|^{D}
		+ \sum_{Q\in V_B'} |f(Q)|^{D}
		+ \sum_{Q\in V_G} |L(f(Q))|^{D} .
		\]
		But, by definition of blue cubes and \eqref{eq: blue desc},
		\[
		\sum_{Q\in V_B'} |f(Q)|^{D} \leq
		\sum_{Q\in V_B'} \delta_Y^{D} \leq N_d \; \text{card} (V_R )\;\delta_Y^{D},
		\]
		and
		\[
		\sum_{Q\in V_G} |L(f(Q))|^{D} \leq
		2 \max \Big\{ c_{u,0}\,\varepsilon\,\delta^{-d/\theta},\;
		N_d\; \text{card}V_R \Big\}\, \delta_Y^{D/\theta},
		\]
		which by the above imply
		\begin{equation}\label{eq: before Holder and VR}
			\mathcal{S}
			\leq \sum_{Q\in V_B^{\text{orig}}} |f(Q)|^{D}
			+ N_d \,\text{card}(V_R) \, \delta_Y^{D}
			+ 2 \max \Big\{ c_{u,0}\,\varepsilon\,\delta^{-d/\theta},\;
			N_d\, \text{card}V_R \Big\} \delta_Y^{D/\theta} .
		\end{equation}
		
		We estimate $\sum_{Q\in V_B^{\text{orig}}} |f(Q)|^{D}$ using $f\in CH^{p,\alpha}$, the defining properties of
		compactly–Hölder maps, and Hölder’s inequality,
		noting that $p/D=\alpha p/d+1>1$, and
		\[
		\frac{1}{p/D} + \frac{1}{\tfrac{p}{p-D}} = 1 .
		\]
		Namely,
		\begin{align}
        \sum_{Q\in V_B^{\text{orig}}} |f(Q)|^{D}
        &\leq \sum_{Q\in V_B^{\text{orig}}} |f(B(Q))|^{D} \nonumber \\
        &\leq \sum_{Q\in V_B^{\text{orig}}} C_Q^{D} |B(Q)|^{D\alpha} \nonumber \\
        &\lesssim \sum_{Q\in V_B^{\text{orig}}} C_Q^{D} |Q|^{D \alpha} \nonumber \\
        &\leq
        \left(\sum_{Q\in V_B^{\text{orig}}} C_Q^p\right)^{D/p}
        \left(\sum_{Q\in V_B^{\text{orig}}} |Q|^{D \alpha  \tfrac{p}{p-D}}\right)^{\tfrac{p-D}{p}} .
        \label{eq:7}
        \end{align}
		
		The comparability constant above \eqref{eq:7} depends only on the  fixed constants $p, \alpha, d'$ and the uniform constants $c_0, C_0, b$ from Theorem~\ref{thm:Dydadic}. Note that
        $$
        D \alpha  \frac{p}{p-D}= \frac{\alpha p^2d' (\alpha p +d')^{-1}}{p-(\alpha p +d')^{-1}p d'}=d',
        $$ and
		
		\[
		\sum_{Q\in V_B^{\text{orig}}} |Q|^{d'}
		= \sum_{Q\in V_B^{\text{orig}}} \Big(\tfrac{|Q|}{\delta}\Big)^{d'} \delta^{d'}
		\leq \delta^{d'} \sum_{Q\in V_B^{\text{orig}}} \Big(\tfrac{|Q|}{\delta}\Big)^d,
		\]
		due to \eqref{eq:1} and $d'>d$.
		Thus, \eqref{eq:7} yields
		\[
		\sum_{Q\in V_B^{\text{orig}}} |f(Q)|^{D}
		\leq (C_E)^{D/p}\;
		\left(\delta^{\,d'-d}\,\sum_{Q\in V_B^{\text{orig}}} |Q|^d\right)^{(p-D)/p}\lesssim \delta^{\tfrac{(d'-d)(p-D)}{p}}\, \eps^{\tfrac{p-D}{p}},
		\] where the comparability constant similarly only depends on uniformly fixed constants.
		Since $d'>d$ and $p-D=\alpha p^2(\alpha p+d')^{-1}>0$, and $\delta< 1$, the above gives
		\begin{equation}\label{eq:8}
			\sum_{Q\in V_B^{\text{orig}}} |f(Q)|^{D} \leq \varepsilon^{(p-D)/p}.
		\end{equation}
		
		For the remaining terms in \eqref{eq: before Holder and VR}, it is enough to
		bound the number of red vertices of $\mathcal{G}$.		
		Let $M(m)$ denote the number of red vertices in the
		$m$-th row of $\mathcal{G}$.

		Let $\{Q_j^m\}_{j\in J_m}$ be all red $m$-level cubes in $V$.
		If $Q_j^m$ is such a cube, by $f\in CH^{p,\alpha}$ and \eqref{eq: CH-inequality} we have
		\[
		|f(Q_j^m)| \leq |f(B(Q_j^m))| \leq C_{Q_j^m}\,|B(Q_j^m)|^\alpha.
		\]
		
		But by $|f(Q_j^m)|\geq \delta_Y$, and $|B(Q_j^m)|\simeq |Q_j^m|\simeq b^m$, due to Theorem~\ref{thm:Dydadic} (iii) and uniform perfectness of $X$, we get
		\begin{equation}\label{eq:9}
		\delta_Y^p \lesssim C_{Q_j^m}^p \; b^{m\alpha p}.
		\end{equation}

		Applying this to all level $m$ red vertices and summing \eqref{eq:9} over all of them gives
		\[
		M(m) \lesssim \delta_Y^{-p}\sum_{j\in J_m} C_{Q_j^m}^p \, b^{m\alpha p}\leq C_E \,b^{m\alpha p}\,\delta_Y^{-p}.
		\]
		
		Summing the above over all levels $m\geq m_1$ yields
		\[
		\text{card} V_R = \sum_{m= m_1}^{m_{K}} M(m)
		\lesssim C_E \,\delta_Y^{-p}\,\sum_{m\geq m_1} b^{m\alpha p}
		\lesssim \delta_Y^{-p}\,b^{m_1\alpha p}.
		\]
        But by $k_1=m_1$ and \eqref{eq:1} we have
		\[
		\text{card} V_R \lesssim \delta_Y^{-p}\,\delta^{\alpha p}.
		\]
		
		Using the above and \eqref{eq:8} on \eqref{eq: before Holder and VR} yields
		\[
		\mathcal{S}
		\lesssim \varepsilon^{(p-D)/p}
		+ N_d\,\delta_Y^{-p}\,\delta^{\alpha p}\,\delta_Y^{D}
		+ 2\max\left\{ c_{u,0}\,\varepsilon\,\delta^{-d/\theta},\; N_d\,\delta_Y^{-p}\,\delta^{\alpha p} \right\}\,\delta_Y^{D/\theta} .
		\]
		
		By choice of $\delta_Y=\delta^{d/D}$ in \eqref{eq: δ in f(E)}, we have
		\[
		N_d\,\delta_Y^{-p}\,\delta^{\alpha p}\,\delta_Y^{D} = N_d\,\delta^{-dp/D}\,\delta^{\alpha p}\,\delta^{d},
		\]
        and because $-dp/D + \alpha p + d = (1-d/d')\alpha p > 0$,
		and
		\[
		-dp/D + \alpha p + d/\theta = (1-d/d')\alpha p + (1/\theta - 1)d > 0 ,
		\] by $\delta\leq \varepsilon$, we get
		\begin{equation}\label{eq:10}
			\mathcal{S} \lesssim
			\varepsilon^{(p-D)/p}
			+ N_d\,\varepsilon^{(1-d/d')\alpha p}
			+ 2\max\Big\{ c_{u,0}\,\varepsilon,\;
			N_d\,\varepsilon^{(1-d/d')\alpha p+(1/\theta -1)d}\Big\}.
		\end{equation} Note that the comparability constant above, say $C(\lesssim)$, only depends on  uniformly fixed constants, and does not depend on $\delta$ or $\epsilon$.
		
		Thus, for any $\varepsilon_Y>0$, there is $\varepsilon>0$
		small enough such that the right-hand-side of \eqref{eq:10} times  $C(\lesssim)$ is less than $\varepsilon_Y$.
		Fix $\delta_\varepsilon$ small enough for all the above assumptions to hold,
		so that there is $\delta_{\varepsilon_Y}'=\delta_\varepsilon^{d/D'}\in(0,1)$ such that
		for every $\delta_Y=\delta^{d/D}\leq \delta_{\varepsilon_Y}'$ there is a cover $\mathcal{U}$
		of $f(E)$, resulting from the corresponding  graph of the source $\mathcal{G}(\mathcal{Q_\delta})$,
		with $\delta_Y^{1/\theta}\leq |U|\leq \delta$ for all $U\in\mathcal{U}$, and
		\[
		\sum_{U\in\mathcal{U}} |U|^{D} < \varepsilon_Y .
		\]
		This implies $\dim_\theta f(E)\leq D$, and the proof is complete by taking $d'\rightarrow d_E$.
    \end{proof}

\begin{remark}
~
    \begin{enumerate}
        \item [(i)] We emphasize that the approach we took for the sum over blue cubes with ancestors in $V$, i.e. over $V_B'$, would not work on the sum over blue cubes $V_B^{\text{orig}}$ in $\mathcal{Q}_\delta$. That is because the bound on their number from \eqref{eq:1}, \eqref{eq:2}, the bound on the diameters of images by $\delta_Y$, and \eqref{eq: δ in f(E)} would yield 
        $$
        \sum_{Q\in V_B^{\text{orig}}} |f(Q)|^{D}\leq c_{u,0}\varepsilon\delta^{-d/\theta}\delta_Y^D=c_{u,0}\varepsilon\delta^{-d/\theta}\delta^{d},
        $$ with the right-hand side being potentially very large for small $\delta$.

        \item [(ii)] Comparing our terminology to that of Kaufman, red and blue cubes correspond to ``major'' and ``minor'' cubes in \cite{Kaufman}. Due to the nature of the intermediate dimension, requiring a lower bound on the diameters of covering sets, we had to introduce the class of green cubes to account for sets in the target that are not too large, but are \textit{too small}. This is not necessary when investigating the distortion of the upper box-counting dimension, and thus a similar notion was not needed in \cite{Kaufman}. We also note that ``critical'' in \cite{Kaufman} corresponds to ``red cubes with blue or green descendants'' in our language, which is, however, not a necessary notion for our argument.

    \end{enumerate}
\end{remark}

\section{Intermediate dimensions under fractionally smooth Sobolev mappings}\label{sec: Sob are cH}
\subsection{Morrey embedding theorem for fractional Haj\l asz--Sobolev, Haj\l asz--Triebel--Lizorkin, and Haj\l asz--Besov spaces}
\label{subsec:Morrey}
In order to show that continuous mappings with a finite Haj\l asz--Triebel--Lizorkin or Haj\l asz--Besov semi-norm are compactly H\"older, we will employ the Morrey embedding theorem.

We begin by establishing such an embedding theorem for
Haj\l asz--Sobolev spaces $\dot{M}^{s,p}$.
To facilitate the formulation of the result, we introduce the
following piece of notation:\  Given $Q,b\in(0,\infty)$, $\sigma\in[1,\infty)$,
and a ball $B_0\subset X$ of radius $R_0\in(0,\infty)$, the measure $\mu$ is
said to satisfy the \emph{$V(\sigma B_0,Q,b)$
condition}\footnote{This condition is a slight variation of the one in \cite[p.\,197]{hajlasz}.},
provided that, for any $x\in X$ and $r\in(0,\sigma R_0]$
satisfying $B(x,r)\subset\sigma B_0$,
\begin{equation}
\label{measbound}
\mu\left(B(x,r)\right)\geq br^Q.
\end{equation}
Note that both locally $Q$-homogeneous and $Q$-Ahlfors regular measures
satisfy the $V(\sigma B_0,Q,b)$ condition for all balls $B_0$ having sufficiently small measure.

\begin{theorem}
\label{embedding}
Let $(X,d_X,\mu)$ be a metric measure space and $(Y,d_Y)$ be a metric space.	
Let $\sigma\in(1,\infty)$, $B_0\subset X$ be  a 
ball of radius $R_0\in(0,\infty)$, and suppose $s,p,Q\in(0,\infty)$ satisfy $sp>Q$.
Assume that the measure $\mu$
satisfies the  $V(\sigma B_0,Q,b)$ condition for some $b\in(0,\infty)$, and 
suppose  $f\in\dot{M}^{s,p}(\sigma B_0:Y)$. 
Then, there exists
a set $N\subset X$ with $\mu(N)=0$, such that,
for all $x,y\in B_0\setminus N$,
\begin{equation}
\label{eq30}
d_Y(f(x),f(y))\leq C b^{-1/p}[d_X(x,y)]^{s-Q/p}\Vert f\Vert_{\dot{M}^{s,p}(\sigma B_0:Y)},
\end{equation}
where $C\in(0,\infty)$ is a constant depending only on
$s$, $p$, $Q$, and $\sigma$. In particular, if $(Y,d_Y)$
is complete then $f$ has a H\"older continuous representative of order $s-Q/p$ on $B_0$, denoted by $u$, satisfying \eqref{eq30}
for all $x,y\in B_0$.
\end{theorem}

While the proof of Theorem~\ref{embedding} is quite similar
to the proof of \cite[Theorem~3.1]{AYY24} and \cite[Theorem~6]{agh20} for $Y=\R$, there were certain adjustments needed to properly establish the result for mappings with an arbitrary metric  space $(Y,d_Y)$ as a target.
\begin{proof}
Choose $g\in D^s(f)\cap L^p(\sigma B_0)$ such that
$\|g\|_{L^p(\sigma B_0)}\approx\Vert f\Vert_{\dot{M}^{s,p}(\sigma B_0:Y)}$.
Without loss of generality, we may assume $\int_{\sigma B_0}g^p\, d\mu>0$.
Indeed, if this integral equals zero,
then $g=0$ $\mu$-almost everywhere in $\sigma B_0$ which, in turn, implies that
$u$ is a constant function $\mu$-almost everywhere in $B_0$ and the result is obvious
in this scenario.

By replacing, if necessary, $g$ with
$\widetilde{g}:=g+ (\mint{-}_{\sigma B_0} g^p\, d\mu )^{1/p}$, we may further assume that
\begin{equation}
\label{eq4}
g(x)\geq 2^{-(1+1/p)}\left(\, \mint{-}_{\sigma B_0} g^p\,  d\mu\right)^{1/p}>0
\
\text{for almost every $x\in\sigma B_0$.}
\end{equation}
Let $N:=E\cup\{x\in\sigma B_0: g(x)=\infty\mbox{ or } g(x)=0\}$, where $E\subset \sigma B_0$ is a set of measure zero such that the pointwise inequality
\eqref{fracHajlasz} holds for every $x,y\in \sigma B_0\setminus E$. Then $N$ is a measurable set satisfying $\mu(N)=0$.

To prove \eqref{eq30}, we will first show that there exists a point $\xi_0\in \sigma B_0\setminus N$  such that, 
for every $x\in B_0\setminus N$,
\begin{equation}
\label{eq29}
d_Y(f(x),f(\xi_0))\leq
Cb^{-1/p}R_0^{s-Q/p}\left(\,\int_{\sigma B_0}g^{p}\, d\mu\right)^{1/p},
\end{equation}
where $C$ is a positive constant depending only on $d$, $s$, $p$, $Q$, and $\sigma$. To this end, for any $k\in\mathbb{Z}$, let
$$
E_k:=\lf\{ x\in\sigma B_0\setminus N:\, g(x)\leq 2^k\r\}.
$$
Clearly  $E_{k-1}\subset E_k$ for any $k\in\mathbb{Z}$. By \eqref{eq4},  we find that
\begin{equation}
\label{eq42}
\bigcup_{k\in\mathbb{Z}} \left[E_k\setminus E_{k-1}\right]=\sigma B_0\setminus N.
\end{equation}
It follows from the pointwise inequality \eqref{fracHajlasz} that $f$ restricted to $E_k$ is $2^{k+1}$-H\"older continuous of order $s$, that is,
\begin{equation}
\label{eq8}
d_Y(f(x),f(y))\leq 2^{k+1}[d_X(x,y)]^s,
\ \forall\,x,y\in E_k.
\end{equation}
Also, applying the Chebyshev inequality, we have
\begin{equation}
\label{eq7}
\mu(\sigma B_0\setminus E_k) = \mu\left(\left\{x\in\sigma B_0:\, g(x)>2^k\right\}\right)\leq 2^{-kp}\int_{\sigma B_0} g^p\, d\mu.
\end{equation}
For any $k\in\mathbb{Z}$ and $\gamma\in Y$, let
\begin{eqnarray}
\label{eq43}
a_k:=\left\{
\begin{array}{ll}
\sup\limits_{E_k\cap B_0} d_Y(f,\gamma)
\quad&
\text{if }\, E_k\cap B_0\neq\emptyset,\\[15pt]
\qquad0
&\text{otherwise.}
\end{array}
\right.
\end{eqnarray}
Clearly, $a_k\leq a_{k+1}$ for any $k\in\mathbb{Z}$.

We will need the following elementary result 
from \cite[Lemma~8]{agh20}.
\begin{lemma}
\label{joasia}
For any $x\in X$ and $r\in(0,\infty)$, if $B(x,r)\subset\sigma B_0$
and $\mu(B(x,r))\geq 2\mu(\sigma B_0\setminus E_k)$ for some $k\in\mathbb{Z}$, then
$$
\mu(B(x,r)\cap E_k)\geq\frac{1}{2}\mu(B(x,r))>0.
$$
\end{lemma}
Moving on, let $k_0$ be the least integer such that
\begin{equation*}
2^{k_0}\geq \left[\frac{2^{1/Q}}{(\sigma-1)(1-2^{-p/Q})}\right]^{Q/p}
\left(bR_0^Q\right)^{-1/p}\left(\, \int_{\sigma B_0} g^p\, d\mu\right)^{1/p}
\end{equation*}
or, equivalently,
\begin{equation}
\label{eq2}
2^{-k_0p/Q}\,\frac{2^{1/Q}b^{-1/Q}}{\left(1-2^{-p/Q}\right)}\,\left(\,\int_{\sigma B_0} g^p\,d\mu\right)^{1/Q}\leq (\sigma-1)R_0.
\end{equation}
Clearly,
\begin{equation}
\label{eq16}
2^{k_0}\approx \left(bR_0^Q\right)^{-1/p}\left(\, \int_{\sigma B_0} g^p\, d\mu\right)^{1/p},
\end{equation}
where the positive equivalence constants depend on $Q$, $p$, $\sigma$, and $d$.
The following lemma is a straightfoward modification
of \cite[Lemma~9]{agh20}.

\begin{lemma}
\label{piotr}
Under the above assumptions, one has  $\mu(E_{k_0})\geq\mu(\sigma B_0)/2$.
\end{lemma}
\begin{proof}
Suppose to the contrary that $\mu(E_{k_0})<\mu(\sigma B_0)/2$. Then
\begin{equation}
\label{eq41}
\mu(\sigma B_0\setminus E_{k_0})>\mu(\sigma B_0)/2.
\end{equation}
By \eqref{eq7} and \eqref{eq2}, we find that
\begin{align*}
r:=& \, 2^{1/Q}b^{-1/Q}[\mu(\sigma B_0\setminus E_{k_0})]^{1/Q}\\
\leq& \,2^{1/Q} b^{-1/Q} 2^{-k_0p/Q}\left(\,\int_{\sigma B_0} g^p\, d\mu\right)^{1/Q}\\
\leq& \,(\sigma-1)(1-2^{-p/Q}) R_0<(\sigma-1)R_0.
\end{align*}
Therefore, if $z_0\in X$ is the center of the ball $B_0$, then $B(z_0,r)\subset\sigma B_0$, so the $V(\sigma B_0,Q,b)$ condition and \eqref{eq41} give 
\begin{align*}
\mu(\sigma B_0)\geq \mu(B(z_0,r))
\geq br^Q=2\mu(\sigma B_0\setminus E_{k_0})>\mu(\sigma B_0),
\end{align*}
which is an obvious contradiction. This finishes the proof of Lemma~\ref{piotr}.
\end{proof}

Now we prove that, for any integer  $k>k_0$,
\begin{equation}
\label{eq44}
a_k\lesssim b^{-s/Q}\left(\,\int_{\sigma B_0} g^p\, d\mu\right)^{s/Q}\,
\sum_{j=k_0}^{k-1} 2^{j(1-sp/Q)} + \sup_{E_{k_0}}d_Y(f,\gamma).
\end{equation}
To see this, let $k\in\mathbb{Z}$ satisfy $k>k_0$. Observe that, if $E_k\cap B_0=\emptyset$, then
 $a_k=0$ due to its definition  and \eqref{eq44} is trivially true.
So we only need to consider the case  $E_k\cap B_0\neq\emptyset$ below. In this case, $a_k=\sup_{E_k\cap B_0} d_Y(f,\gamma)$. Now, for any $i\in \{0,1,\ldots,k-k_0-1\}$,  define
$$
r_{k-i}:=2^{1/Q}b^{-1/Q}2^{-[k-(i+1)]p/Q}\left(\,\int_{\sigma B_0}g^p\, d\mu\right)^{1/Q}.
$$
Then,  by \eqref{eq2}, we conclude that
\begin{align}
\label{michal}
&r_k+r_{k-1}+\ldots+r_{k_0+1}\nonumber\\
&\quad=
2^{1/Q}b^{-1/Q}\left(\,\int_{\sigma B_0} g^p\, d\mu\right)^{1/Q}\left(\sum_{i=0}^{k-k_0-1} 2^{-[k-(i+1)]p/Q}\right)\nonumber\\
&\quad<
2^{-k_0 p/Q}\frac{2^{1/Q}b^{-1/Q}}{\lf(1-2^{-p/Q}\r)}\left(\,\int_{\sigma B_0} g^p\, d\mu\right)^{1/Q}
\leq (\sigma-1)R_0.
\end{align}
The importance of \eqref{michal} will reveal itself shortly. Since $E_k\cap B_0\neq\emptyset$, we can choose a point $x_k\in E_k\cap B_0$ arbitrarily. We   now use induction with respect to $i$ to define a sequence $x_{k-i}\in\sigma B_0$, $i\in\{1,\ldots,k-k_0\}$, such that $x_{k-1}\in E_{k-1}\cap B(x_k,r_k)$,
$$
x_{k-2}\in E_{k-2}\cap B(x_{k-1},r_{k-1}),\ \ldots,\
x_{k_0}\in E_{k_0}\cap B(x_{k_0+1},r_{k_0+1}).
$$
To be precise, for $i=1$,  we choose $x_{k-1}$ as follows. First,
by \eqref{michal}, we find that, for any point $y\in B(x_k,r_k)$,
\begin{align*}
d_X(z_0,y)
\leq d_X(z_0,x_k)+d_X(x_k,y)<R_0+r_k<R_0+(\sigma-1)R_0=\sigma R_0,
\end{align*}
which   implies $B(x_k,r_k)\subset \sigma B_0$.
 As such, applying the $V(\sigma B_0,Q,b)$ condition  and  \eqref{eq7}, we conclude that
$$
\mu(B(x_k,r_k))\geq br_k^Q=2\; 2^{-(k-1)p}\int_{\sigma B_0} g^p\, d\mu\geq2\mu(\sigma B_0\setminus E_{k-1}),
$$
which, together with  Lemma~\ref{joasia},
further  implies that $\mu (E_{k-1}\cap B(x_k,r_k))>0$ and hence
we can find a point $x_{k-1}\in E_{k-1}\cap B(x_k,r_k)$.
Clearly $x_{k-1}\in\sigma B_0$.
If $k-k_0=1$, then we are done, so suppose that $k-k_0>1$
and assume that we already selected points
$x_{k-1},\ldots, x_{k-i}$ for some $1\leq i<k-k_0$ satisfying
\begin{equation*}
x_{k-j}\in \sigma B_0\cap E_{k-j}\cap B(x_{k-j+1},r_{k-j+1})
\ \mbox{ for\ any\ $j\in\{1,\ldots,i\}$}.
\end{equation*}
It remains to  select
$$
x_{k-(i+1)}\in \sigma B_0\cap E_{k-(i+1)}\cap B(x_{k-i},r_{k-i}).
$$
By  \eqref{michal}, we find that, for any $y\in B(x_{k-i},r_{k-i})$,
\begin{align*}
d_X(y,x_k)&
\leq d_X(y,x_{k-i})+d(x_{k-i},x_{k-i+1})+\ldots+ d_X(x_{k-1},x_k)\\
&< r_{k-i}+r_{k-i+1}
+\ldots+r_k\leq (\sigma-1)R_0,
\end{align*}
which, together with $x_k\in B_0$, further  implies that
\begin{align*}
d_X(z_0,y)&\leq d_X(z_0,x_k)+d_X(x_k,y)<R_0+(\sigma-1)R_0=\sigma R_0.
\end{align*}
Thus, $B(x_{k-i},r_{k-i})\subset\sigma B_0$, and then the $V(\sigma B_0,Q,b)$ condition  and  \eqref{eq7} imply
\begin{align*}
\mu(B(x_{k-i},r_{k-i}))&\geq br_{k-i}^Q\\
&=2\; 2^{-[k-(i+1)]p}\int_{\sigma B_0} g^p\, d\mu\geq2\mu(\sigma B_0\setminus E_{k-(i+1)}).
\end{align*}
Applying Lemma~\ref{joasia}, we have $\mu(E_{k-(i+1)}\cap B(x_{k-i},r_{k-i}))>0$ and  can find a point
$$
x_{k-(i+1)}\in E_{k-(i+1)}\cap B(x_{k-i},r_{k-i}).
$$
Clearly $x_{k-(i+1)}\in\sigma B_0$. That finishes the inductive argument on the choose
of $\{x_{k_0},\ldots,x_{k-1}\}$.

Notice that, for any $i\in\{0,1,\ldots,k-k_0-1\}$,
\begin{align*}
d_X(x_{k-i},x_{k-(i+1)})<r_{k-i}
&=2^{1/Q}b^{-1/Q}2^{-[k-(i+1)]p/Q}\left(\,\int_{\sigma B_0} g^p\, d\mu\right)^{1/Q}.
\end{align*}

Also, recall that, by \eqref{eq8}, $f$ restricted to $E_{k-i}$ is $2^{k-i+1}$-H\"older continuous of
order $s$. From these and the fact that $x_{k-i},x_{k-(i+1)}\in E_{k-i}$
and $x_{k_0}\in E_{k_0}$, it follows that
\begin{align}
\label{eq45}
d_Y(f(x_k),\gamma)
&\leq
\sum_{i=0}^{k-k_0-1}d_Y(f(x_{k-i}),f(x_{k-(i+1)})) + d_Y(f(x_{k_{0}}),\gamma)\nonumber\\
&\leq
\sum_{i=0}^{k-k_0-1} 2^{k-i+1}[d_X(x_{k-i},x_{k-(i+1)})]^s + \sup_{E_{k_0}}d_Y(f,\gamma)\nonumber\\
&\lesssim 2^{s/Q}b^{-s/Q}\left(\,\int_{\sigma B_0} g^p\, d\mu\right)^{s/Q}\,
\sum_{i=0}^{k-k_0-1} 2^{[k-(i+1)](1-sp/Q)}+\sup_{E_{k_0}}d_Y(f,\gamma)\nonumber\\
&\lesssim b^{-s/Q}\left(\,\int_{\sigma B_0} g^p\, d\mu\right)^{s/Q}\,
\sum_{j=k_0}^{k-1} 2^{j(1-sp/Q)} + \sup_{E_{k_0}}d_Y(f,\gamma).
\end{align}
Since $x_k\in E_k\cap B_0$ was selected arbitrarily, taking the supremum in \eqref{eq45} over all $x_k\in E_k\cap B_0$, we  obtain the desired estimate in \eqref{eq44}.

Observe that, on the one hand,   for any $x,y\in\sigma B_0$, 
\begin{equation}\label{estdia}
|\sigma B_0|\leq  2\sigma R_0.
\end{equation}
On the other hand, by Lemma~\ref{piotr} we have $\mu(E_{k_0})>0$, and so 
we can choose a point $\xi_0\in E_{k_0}\subset \sigma B_0\setminus N$. 
Taking $\gamma:=f(\xi_0)$ and applying the H\"older continuity \eqref{eq8}, gives
\begin{align}\label{eq17}
\sup_{E_{k_0}}d_Y(f,f(\xi_0))
&\leq 2^{k_0+1}|\sigma B_0|^s \leq
2^{k_0+1}\left[ 2\sigma R_0\right]^s\nonumber\\
&\lesssim R_0^s\left(bR_0^Q\right)^{-1/p}\left(\, \int_{\sigma B_0} g^p\, d\mu\right)^{1/p},
\end{align}
where the implicit positive
constant depends only on $Q$, $p$, $\sigma$, $s$, and $d$.

With $\gamma=f(\xi_0)$ as in \eqref{eq17} and $\{a_k\}_{k\in\mathbb{Z}}$
as in \eqref{eq43}, observe that since $2^{1-sp/Q}<1$, by \eqref{eq44},
we find that,  for any integer  $k> k_0$,
\begin{equation}
\label{eq28}
a_k\ls b^{-s/Q}\left(\, \int_{\sigma B_0} g^p\, d\mu\right)^{s/Q} 2^{k_0(1-sp/Q)}
+\sup_{E_{k_0}}d_Y(f,f(\xi_0)).
\end{equation}
However, since, for every $k\leq k_0$,
$$
a_k:=\sup_{E_k\cap B_0}d_Y(f,f(\xi_0))\leq \sup_{E_{k_0}}d_Y(f,f(\xi_0)),
$$ 
we actually have that \eqref{eq28} holds for all $k\in\mathbb{Z}$. From this, \eqref{eq16}, and \eqref{eq17}, we deduce that,  for every $k\in\mathbb{Z}$,
\begin{equation}\label{kysr}
a_k:=\sup_{E_k\cap B_0}d_Y(f,f(\xi_0))\ls b^{-1/p}R_0^{s-Q/p}\left(\, \int_{\sigma B_0}g^p\, d\mu\right)^{1/p}.
\end{equation}
Noticing that  the right-hand side of \eqref{kysr}  is a positive
constant independent of $k$, by \eqref{eq42} and
the definition of $a_k$ in \eqref{eq43}, we conclude that the function
$d_Y\left(f,f(\xi_0)\right)$ is bounded on $B_0\setminus N$ by the right-hand side of \eqref{kysr}
(modulo a positive constant). 
This proves \eqref{eq29}.

Next we show  the  H\"older continuity of $u$ along with the estimate \eqref{eq30}.
To this end, fix $x,y\in B_0\setminus N$. If $2d_X(x,y)\leq (\sigma-1)R_0/\sigma$,   let
$R_1:=2d_X(x,y)$. Clearly,   $x,y\in B_1:=B(x,R_1)$. Moreover,
if $z_0$ denotes the center of $B_0$, then, for any $z\in \sigma B_1$, we have
\begin{align*}
d_X(z_0,z)&\leq d_X(z_0,x)+d_X(x,z)< R_0+\sigma R_1
<R_0+(\sigma-1)R_0=\sigma R_0,
\end{align*}
which implies that $\sigma B_1\subset\sigma B_0$.
From this and the assumption that $\mu$ satisfies the $V(\sigma B_0,Q,b)$ condition,
it follows that $\mu$ also satisfies the $V(\sigma B_1,Q,b)$ condition and therefore,
by the estimate \eqref{eq29} applied to $B_1$ in place of $B_0$, we can find a point
$\xi_1\in\sigma B_1\setminus N$ such that,
\begin{align*}
d_Y(f(x),f(y))&\leq d_Y(f(x),f(\xi_1))+d_Y(f(\xi_1),f(y))\\
&\ls b^{-1/p}R_1^{s-Q/p}\, \left(\, \int_{\sigma B_1}g^p\, d\mu\right)^{1/p}\\
&\approx b^{-1/p}[d_X(x,y)]^{s-Q/p}\left(\, \int_{\sigma B_1} g^p\, d\mu\right)^{1/p}\\
&\ls b^{-1/p}[d_X(x,y)]^{s-Q/p}\left(\, \int_{\sigma B_0} g^p\, d\mu\right)^{1/p}.
\end{align*}
If $2d_X(x,y)>(\sigma-1)R_0/\sigma$, then \eqref{eq29} immediately gives 
\begin{align*}
d_Y(f(x),f(y))
&\leq d_Y(f(x),f(\xi_0))+d_Y(f(\xi_0),f(y))\\
&\ls b^{-1/p}R_0^{s-Q/p}\, \left(\, \int_{\sigma B_0}g^p\, d\mu\right)^{1/p}\\
&\ls b^{-1/p}[d_X(x,y)]^{s-Q/p}\left(\, \int_{\sigma B_0} g^p\, d\mu\right)^{1/p}.
\end{align*}
Since $\|g\|_{L^p(\sigma B_0)}\approx\Vert f\Vert_{\dot{M}^{s,p}(\sigma B_0:Y)}$, these estimates imply that \eqref{eq30} holds for every $x,y\in B_0\setminus N$. Since $B_0\setminus N$ is dense in $B_0$ [here we are using the fact that $\mu$ is positive and finite on balls], we can rely on the completeness of $(Y,d_Y)$ to extend $f$ to a H\"older continuous function of order $s-Q/p$ on $B_0$, that satisfies  \eqref{eq30} for every $x,y\in B_0$. This completes the proof of Theorem~\ref{embedding}.
\end{proof}

\begin{remark}\label{Rmk:continuous}
It is easy to see from the proof of Theorem~\ref{embedding} that if $f:\sigma B_0\to Y$ is a continuous mapping such that $\Vert f\Vert_{\dot{M}^{s,p}(\sigma B_0:Y)}<\infty$ then \eqref{eq30} holds for all $x,y\in B_0$, without needing to assume that $(Y,d_Y)$ is complete.
\end{remark}

We now establish the Morrey embedding theorem for Haj\l asz--Triebel--Lizorkin and Haj\l asz--Besov on locally $Q$-homogeneous metric measure spaces. We will need the following result, which was proven for real-valued functions  in \cite[Proposition~2.4]{AYY24}. The same proof is also valid for metric space-valued functions and therefore we omit the details.
\begin{lemma}
\label{sobequal}
Let $(X,d_Y,\mu)$ be a metric space
equipped with a nonnegative Borel measure $\mu$, $(Y,d_Y)$ be a metric space, and
suppose that $s,p\in(0,\infty)$ and $q\in(0,\infty]$. Then
\begin{enumerate}
\item $\dot{M}^s_{p,q}(X:Y)\hookrightarrow\dot{M}^s_{p,\infty}(X:Y)$;

\item $\dot{M}^s_{p,\infty}(X:Y)=\dot{M}^{s,p}(X:Y)$ as sets, with equal semi-norms;

\item if $q\in(0,p]$, then $\dot{N}^s_{p,q}(X:Y)\hookrightarrow \dot{M}^{s,p}(X:Y)\hookrightarrow\dot{N}^s_{p,\infty}(X:Y)$;

\item for any $\varepsilon\in(0,s)$, there exists a  constant $C\in(0,\infty)$ such that, if $B\subset X$ is a ball with radius $r\in(0,\infty)$, then $\dot{N}^s_{p,q}(B:Y)\hookrightarrow\dot{M}^{\varepsilon,p}(B:Y)$, where $\Vert f\Vert_{\dot{M}^{\varepsilon,p}(B:Y)}\leq Cr^{s-\varepsilon}\Vert f\Vert_{\dot{N}^s_{p,q}(B:Y)}$ for all $f\in \dot{N}^s_{p,q}(B:Y)$.
\end{enumerate}
\end{lemma}

\begin{corollary}
\label{DOUBembedding-cpt}
Let $(X,d,\mu)$ be a metric measure space, where $\mu$ is  locally $Q$-homogeneous for some $Q\in(0,\infty)$, and let $(Y,d_Y)$ be any metric space. Suppose $s\in(0,\infty)$, $p\in(Q/s,\infty)$, and $q\in(0,\infty]$,
and assume $f:X\to Y$ is a continuous function. 
Then, for any compact set $K\subset X$, there exist constants $C_K\in[1,\infty)$ and $R_K\in(0,\infty)$, both of which are independent of $f$, such that, for all
balls $B_0:=B(x_0,R_0)$ with $x_0\in K$ and
$R_0\in(0,R_K)$ for which $\Vert f\Vert_{\dot{M}^s_{p,q}(2B_0:Y)}<\infty$, one has
\begin{equation}
\label{eq30-DOUB-cpt}
|f(x)-f(y)|\leq C_K[d(x,y)]^{s-Q/p}\frac{R_0^{Q/p}}
{[\mu(2B_0)]^{1/p}}\,\Vert f\Vert_{\dot{M}^s_{p,q}(2B_0:Y)},
\end{equation}
for all $x,y\in B_0$. The above statement is also
valid with  the Haj\l asz--Triebel--Lizorkin space $\dot{M}^s_{p,q}$ replaced by the Haj\l asz--Besov space $\dot{N}^s_{p,q}$.
\end{corollary}

\begin{remark}
In Corollary~\ref{DOUBembedding-cpt}, the specific dilation factor $2$ of the ball $B_0$
is not essential and can be replaced by any $\sigma\in(1,\infty)$.    
\end{remark}

\begin{proof}
Fix a compact set $K\subset X$. 
We claim that there exists a constant $\kappa\in(0,\infty)$ such that 
\begin{equation}\label{Qdoub-cpt2}
\kappa\left(\frac{r_1}{r_2}\right)^{Q}\leq\frac{\mu(B(y,r_1))}{\mu(B(x,r_2))},
\end{equation}
for all $x,y\in K$ and  $0<r_1<r_2<\tilde{R}_{\text{hom}}(K)/2$
satisfying $B(y,r_1)\subset B(x,r_2)$, where $\tilde{R}_{\text{hom}}(K)\in(0,\infty)$ is as in the $Q$-homogeneous condition \eqref{Qhom}.
Suppose $x,y\in K$ and $0<r_1<r_2<\tilde{R}_{\text{hom}}(K)/2$
are such that $B(y,r_1)\subset B(x,r_2)$. Since $0<r_1<2r_2<\tilde{R}_{\text{hom}}(K)$, by  \eqref{Qhom} we have
$$
\frac{\mu(B(y,2r_2))}{\mu(B(y,r_1))}\leq \tilde{C}_{\text{hom}}(K) \left(\frac{2r_2}{r_1}\right)^Q,
$$
which, combined with the observation $B(x,r_2)\subset B(y,2r_2)$,
gives
$$
\mu(B(x,r_2))\leq\mu(B(y,2r_2))\leq\tilde{C}_{\text{hom}}(K)
\left(\frac{2r_2}{r_1}\right)^{Q}\mu(B(y,r_1)).
$$
Thus, \eqref{Qdoub-cpt2} holds with $\kappa:=[2^Q\tilde{C}_{\text{hom}}(K)]^{-1}$.

Moving on, let $R_K:=\tilde{R}_{\text{hom}}(K)/4$ and 
$B_0:=B(x_0,R_0)$ a ball with $x_0\in K$ and $R_0\in(0,R_K)$.
Then \eqref{Qdoub-cpt2} applied with $B(x,r_2):=B(x_0,2R_0)=2B_0$ implies that the  measure $\mu$ satisfies the $V(2B_0, Q, b)$ condition
with $b:=\kappa\mu(2B_0)(2R_0)^{-Q}$. 
On the other hand,  it follows from (1) and (2) in Lemma~\ref{sobequal} that $\Vert f\Vert_{\dot{M}^{s,p}(2B_0:Y)}\lesssim \Vert f\Vert_{\dot{M}^{s}_{p,q}(2B_0:Y)}$ which, combined with the assumption that  $\Vert f\Vert_{\dot{M}^s_{p,q}(2B_0:Y)}$ is finite, gives that $\Vert f\Vert_{\dot{M}^{s,p}(2B_0:Y)}$ is also finite. Given these observations, by Theorem~\ref{embedding} and Remark~\ref{Rmk:continuous} we conclude that \eqref{eq30} holds for all $x,y\in B_0$ with $b:=\kappa\mu(2B_0)(2R_0)^{-Q}$.  This, together with $\Vert f\Vert_{\dot{M}^{s,p}(2B_0:Y)}\lesssim \Vert f\Vert_{\dot{M}^{s}_{p,q}(2B_0:Y)}$, completes the proof of \eqref{eq30-DOUB-cpt} for $\dot{M}^s_{p,q}$ spaces.

To obtain \eqref{eq30-DOUB-cpt} for $\dot{N}^s_{p,q}$ spaces, choose $\varepsilon\in(0,s)$ close enough to $s$ so that $p\in(Q/\varepsilon,\infty)$. 
By Lemma~\ref{sobequal}(4), 
$\dot{N}^s_{p,q}(2 B_0:Y)\hookrightarrow\dot{M}^{\varepsilon,p}(2B_0:Y)$
with
\begin{equation}
\label{xq-88-cpt}
\Vert f\Vert_{\dot{M}^{\varepsilon,p}(2B_0:Y)}\leq C R_0^{s-\varepsilon}\Vert f\Vert_{\dot{N}^s_{p,q}(2B_0:Y)},
\end{equation}
for some constant $C\in(0,\infty)$ that is independent of $f$ and the ball $B_0$. Given this and the fact that $p\in(Q/\varepsilon,\infty)$, by arguing as we did in the proof of \eqref{eq29} with the function $f\in\dot{M}^{\varepsilon,p}(\sigma B_0:Y)$,  we conclude that
there exist a set $N\subset X$ (which can be chosen independent of $B_0$) with $\mu(N)=0$  and a point
$\xi_0\in 2 B_0\setminus N$  such that, 
for every $x\in B_0\setminus N$,
\begin{align}
\label{eq29X-cpt}
d_Y(f(x),f(\xi_0))&\leq
b^{-1/p}R_0^{\varepsilon-Q/p}\Vert f\Vert_{\dot{M}^{\varepsilon,p}(2B_0:Y)}\nonumber\\
&\lesssim  b^{-1/p}R_0^{s-Q/p}\Vert f\Vert_{\dot{N}^s_{p,q}(2B_0:Y)}.
\end{align}
where the implicit constant depends only on $d$, $s$, $\varepsilon$, $p$, and $Q$. With \eqref{eq29X-cpt} in hand, an argument analogous to the one used at the end Theorem~\ref{embedding} (keeping in mind that $b:=\kappa\mu(2B_0)(2R_0)^{-Q}$) gives that
the inequality in \eqref{eq30-DOUB-cpt} holds pointwise almost everywhere in $B_0$. Given that $f$ is continuous, we conclude that \eqref{eq30-DOUB-cpt} in fact holds pointwise everywhere in $B_0$. This finishes the proof of Corollary~\ref{DOUBembedding-cpt}.
\end{proof}

\subsection{The proof of Theorem~\ref{thm:FractionalIntDim}}
\label{subsec:proofthem1.3}
To prove Theorem~\ref{thm:FractionalIntDim}, we require one final technical lemma, whose statement relies on the following notion:
Given a metric measure space $(X,d_X,\mu)$ and a threshold $R\in(0,\infty)$,
recall that the \emph{restricted Hardy--Littlewood maximal function}
of $f\in L^1_{\loc}(X)$ is defined by setting, for each $x\in X$,
$$
M_Rf(x):=\sup_{r\in(0,R)}\,\mint{-}_{B(x,r)}|f(y)|\, d\mu(y).
$$
Suppose $\mu$ is locally $Q$-homogeneous for some $Q\in(0,\infty)$
and let $\tilde{R}_{\text{hom}}\in(0,\infty)$ be as in the local $Q$-homogeneity condition \eqref{Qhom}. Then for every non-empty compact set $K\subseteq X$ and $p\in(1,\infty)$,
there exists a constant $C'\in(0,\infty)$ such that, 
\begin{equation}\label{maximalestimate}
\Vert M_Rf\Vert_{L^p(K)}\leq C'\Vert f\Vert_{L^p(K')},
\end{equation}
for all $R\in(0,\tilde{R}_{\text{hom}}(K)/10)$ and $f\in L^p(K')$, 
where $K'\subset X$ is any measurable set that contains every ball centered in $K$ with radius at most $\tilde{R}_{\text{hom}}/10$. Indeed, this
follows from arguing as in the proof of the Maximal Function Theorem, see, for instance, \cite[Chapter~2]{hei:lectures} or \cite[Theorem~3.5.6]{Juha-Jeremy-etc-book}. 

We can now state the lemma alluded to above.
\begin{lemma}
\label{Le:Make_disjoint_integrals}
Let $(X,d_X,\mu)$ be a proper metric measure space, where $\mu$ is locally doubling, and fix $p\in(0,\infty)$, $t\in(0,p)$, and $\tau\in (0,1)$. 
For each compact set $K\subset X$ there is a constant $C_K'\in[1,\infty)$ and a radius $R_K'\in(0,\infty)$  such that for all nonnegative functions $g\in L^p_{\loc}(X)$,
\begin{equation}\label{eq:MaximalFunIneq}
\mint{-}_{B(x,r)}g^t\, d\mu \leq C_K' \mint{-}_{B(x,\tau r)}M_{R_K'}(g^t)\, d\mu,
\end{equation} for all $x\in K$ and $0<r<R_K'$.
\end{lemma}
Lemma~\ref{Le:Make_disjoint_integrals} is a corollary of the Maximal Function Theorem, see \cite[Chapter~2]{hei:lectures} and  
\cite[Lemma~3.3]{BaloghAGMS} for a proof.

We plan on using Corollary~\ref{DOUBembedding-cpt} and Lemma~\ref{Le:Make_disjoint_integrals} to show that $f$ belongs to the appropriate compactly H\"older class, which is enough to achieve \eqref{eq: main TL-Besov q<=p Interm bound} and \eqref{eq:SobBoxdistortion}.

\begin{proof}[Proof of Theorem~\ref{thm:FractionalIntDim}]
Let $E\subset X$ be compact and $\eps\in (0,1)$. Note that the centers of balls that cover $E$ in the definition of compactly H\"older mappings do not necessarily lie in $E$, while it is a requirement for the inequalities in Corollary~\ref{DOUBembedding-cpt} and Lemma~\ref{Le:Make_disjoint_integrals}. Thus, we need to apply these properties to a potentially larger compact set than $E$. Take any point $x_E\in E$. Then a straightforward calculation  shows that
\begin{equation}\label{eq:E in comp_ball}
E\cup\left(\bigcup_{x\in E}B(x,1/4)\right) \subset B(x_E,  |E|+1/2).
\end{equation}  Thus, if $K$ is the closure of $B(x_E,  |E|+1/2)$, and $\mathcal{B}$ is a covering of $E$ by balls of radius at most $1/10$, then all elements of $\mathcal{B}$ lie entirely in $K$, along with their centers.

Our plan is to apply Corollary~\ref{DOUBembedding-cpt} and Lemma~\ref{Le:Make_disjoint_integrals} with the set $K$, which is compact  by virtue of $X$ being proper. With this goal in mind,
consider the radius 
\begin{equation}\label{eq:SobIsCH_radius}
    r_E:= \min\left\{ \frac{R_K}{2}, \frac{R_K'}{2}, \frac{R_\text{hom}(K)}{10},\frac{1}{10} \right\} ,
\end{equation} where $R_K, R_K'$ are as in Corollary~\ref{DOUBembedding-cpt} and Lemma~\ref{Le:Make_disjoint_integrals}, respectively,  and $R_{\text{hom}}(K):=\tilde{R}_{\text{hom}}(K)/3$ with $\tilde{R}_{\text{hom}}(K)$ as in the local $Q$-homogeneity condition \eqref{Qhom} for the compact set $K$.
With a view of applying \eqref{maximalestimate} later in the proof, let $K'\subset X$ be any measurable set that 
contains every ball centered in $K$ with radius at most $2r_E$. Since $X$ is proper, we can choose $K'$ to be compact.
 
Suppose $\{B(x_i,r_i)\}_{i\in I}$ is a  cover of $E$ with $r_i<r_E$ and $B(x_i,\eps r_i)\cap B(x_j,\eps r_j)=\emptyset$ for all distinct $i,j\in I$. We will show that an inequality of the form \eqref{eq: CH-def-inequality} holds. Let us first consider the case when $\Vert f\Vert_{\dot{M}^s_{p,q}(X:Y)}$ is finite. By (1) and (2) in Lemma~\ref{sobequal}, we can assume $\Vert f\Vert_{\dot{M}^{s,p}(X:Y)}$ is finite. Then $\Vert f\Vert_{\dot{M}^{s,p}(K':Y)}$ is finite and so there exists $g\in D^s(f)\cap L^p(K')$. 
Fix $t\in(Q/s,p)$. 
By H\"older's inequality and the fact that each $2B_i:=B(x_i,2r_i)\subset K'$, we have 
$\Vert f\Vert_{\dot{M}^{t,p}(2B_i:Y)}$ is finite for each fixed $B_i$, where the pointwise restriction of $g:K'\to[0,\infty]$ to $2B_i$ serves as an $s$-gradient of $f$.
By this, \eqref{eq:SobIsCH_radius}, and Corollary~\ref{DOUBembedding-cpt}, we have that there exists a constant $C_K\in[1,\infty)$, which is independent of $f$, such that, for all balls $B_i:=B(x_i,r_i)$, the following inequality holds
$$
d_Y(f(x),f(y))\leq C_K |B_i|^{Q/t}[d_X(x,y)]^{s-Q/t} \left( \mint{-}_{2B_i} g^t \, d\mu \right)^{1/t},
$$ for all $x,y\in B_i$. Setting $\alpha:=s-Q/t$, we deduce from the preceding inequality and the definition of $|f|_{\alpha,B_i}$,  that
$$
|f|_{\alpha,B_i}^t\leq C_K^t |B_i|^{Q}\mint{-}_{2B_i} g^t\, d\mu.
$$

Note that, by \eqref{eq:lower-AR}, we could uniformly bound the term $ \frac{|B_i|^{Q}}{\mu(2B_i)}$ on the right from above, and just leave the integral terms to depend on $i$. However, setting $C_i$ to be the uniform constant times the integral on the right hand side of the above inequality would not be enough to prove the compactly H\"older property. The reason for this is the potentially large overlap between the balls $2B_i$, contradicting any upper bound $C_E$ on the $t$-sum of the constants $C_i$. To avoid this issue, we apply Lemma~\ref{Le:Make_disjoint_integrals} with $\tau:=\eps/2$ on the integral $\mint{-}_{2B_i} g^t\, d\mu$ (which is possible due to the choice \eqref{eq:SobIsCH_radius} and  $r_i<r_E$) to obtain
$$
|f|_{\alpha,B_i}^t\leq C_K^t C_K' |B_i|^{Q}\mint{-}_{\eps B_i} \tilde{g}\, d\mu,
$$ 
where $\tilde{g}:=M_{r_E}(g^t)\in L^{p/t}(K)\subset L^1(K)$ by \eqref{maximalestimate}. Here we view $g$ as a locally $p$-integrable function defined on $X$ by extending it by zero outside its original domain.
Since the measure $\mu$ is locally $Q$-homogeneous, by \eqref{eq:lower-AR} and \eqref{eq:SobIsCH_radius}, the quotient $\frac{|B_i|^Q}{\mu(\eps B_i)}$ is bounded from above by $\tilde{C}:=C_{\text{hom}}(K) \left(\frac{1}{\eps}\right)^Q$. Hence, 
due to $B(x_i,\eps r_i)\cap B(x_j,\eps r_j)=\emptyset$ and 
$$
\sum_{i\in I} \int_{\eps B_i} \tilde{g}\, d\mu =  \int_{\bigcup\limits_{i\in I} \eps B_i} \tilde{g}\, d\mu \leq  \int_{K} \tilde{g}\, d\mu<\infty,
$$ 
there is a constant $C_E:=C_K^t C_K' \tilde{C} \left( \int_{K} \tilde{g}\, d\mu \right)<\infty$ satisfying
$$
\sum\limits_{i\in I} |f|_{\alpha,B_i}^t \leq \sum\limits_{i\in I} C_K^q C_K' \tilde{C}  \int_{\eps B_i} \tilde{g}\, d\mu  \leq C_E.
$$
Since $E$ and $\eps$ were arbitrary, and given that $K$ depends only on $E$, we deduce that $f\in CH^{t,s-Q/t}$ for all $t\in (Q/s,p)$ as wanted. As such, it follows from Theorem~\ref{thm:main_Holder} that
$$
\dim_\theta f(E) \leq\frac{t d_E(\theta)}{st-Q+d_E(\theta)},
$$ for all $t\in (Q/s,p)$, which, by letting $t\rightarrow p$, gives \eqref{eq: main TL-Besov q<=p Interm bound}.

Suppose that $\Vert f\Vert_{\dot{N}^s_{p,q}(X:Y)}$ is finite. Note that $\Vert f\Vert_{\dot{N}^s_{p,q}(K':Y)}$ is finite and so there exists $\vec{g}:=\{g_k\}_{k\in\mathbb{Z}}\in\mathbb{D}^s(f)\cap\ell^q(L^p(K'))$.
By Lemma~\ref{sobequal}(3), if  $q\in(0,p]$, then $\dot{N}^s_{p,q}(K':Y)\hookrightarrow \dot{M}^{s,p}(K':Y)$ and the desired conclusions follow from what we have already proven. Thus, suppose $q>p$ and fix $t\in(Q/s,p)$.
Similar to as before, we have that  $\Vert f\Vert_{\dot{N}^s_{t,q}(2B_i:Y)}$ is finite  for each fixed $B_i$, where the pointwise restriction of the functions in the sequence $\{g_k\}_{k\in\mathbb{Z}}$ to $2B_i$ serves as a fractional $s$-gradient of $f$. Again, we view each $g_k:K'\to[0,\infty]$ as a locally $p$-integrable function defined on $X$ by extending it by zero outside its original domain.
By this, \eqref{eq:SobIsCH_radius}, and Corollary~\ref{DOUBembedding-cpt},  we have that there exists a constant $C_K\in[1,\infty)$, which is independent of $f$, such that, for all balls $B_i$, the following inequality holds
$$
d_Y(f(x),f(y))\leq C_K |B_i|^{Q/t}[d_X(x,y)]^{s-Q/t} \left(\sum_{k\in\mathbb{Z}}\left( \mint{-}_{2B_i} g_k^t\, d\mu \right)^{q/t}\right)^{1/q},
$$ for all $x,y\in B_i$. Arguing as in the case when $\Vert f\Vert_{\dot{M}^s_{p,q}(X:Y)}$ is finite, we have 
 $$
|f|_{\alpha,B_i}^q\leq C_K^q (C_K'\tilde{C})^{q/t} \sum_{k\in\mathbb{Z}}\left( \int_{\eps B_i} \tilde{g}_k\, d\mu \right)^{q/t},
$$ 
where $\alpha:=s-Q/t$, $\tilde{g}_k:=M_{r_E}(g_k^t)\in L^{p/t}(K)\subset L^1(K)$, and $\tilde{C}:=C_{\text{hom}}(K) \left(\frac{1}{\eps}\right)^Q$ is the same as before.
Since $q>p>t$, we can estimate
\begin{align}\label{syr}
\sum_{i\in I} \sum_{k\in\mathbb{Z}}\left( \int_{\eps B_i} \tilde{g}_k\, d\mu \right)^{q/t}
&=\sum_{k\in\mathbb{Z}}\sum_{i\in I} \left( \int_{\eps B_i} \tilde{g}_k\, d\mu \right)^{q/t}
\leq\sum_{k\in\mathbb{Z}}\left( \sum_{i\in I} \int_{\eps B_i} \tilde{g}_k\, d\mu \right)^{q/t}\nonumber\\
&=
\sum_{k\in\mathbb{Z}}\left(\int_{\bigcup\limits_{i\in I} \eps B_i} \tilde{g}_k\, d\mu\right)^{q/t}
\leq \sum_{k\in\mathbb{Z}}\left(\int_{K} \tilde{g}_k\, d\mu\right)^{q/t}.
\end{align}
On the other hand, by H\"older's inequality (used with $p/t>1$) and \eqref{maximalestimate}, we have
$$
\int_{K} \tilde{g}_k\, d\mu\leq[\mu(K)]^{1-t/p}\left(\int_{K} (\tilde{g}_k)^{p/t}\, d\mu\right)^{t/p}\leq C'
[\mu(K)]^{1-t/p}\left(\int_{K'} g_k^{p}\, d\mu\right)^{t/p},
$$
for some constant $C'\in(0,\infty)$ that only depends on the doubling constant $\tilde{C}_{\text{hom}}$ from \eqref{Qhom}.
This, combined with \eqref{syr} and the fact that 
$$
\sum_{k\in\mathbb{Z}}\left(\int_{K'} g_k^{p}\,d\mu\right)^{q/p}=\Vert\vec{g}\Vert_{\ell^q(L^p(K'))}^q<\infty,
$$
gives
$$
\sum_{i\in I} \sum_{k\in\mathbb{Z}}\left( \int_{\eps B_i} \tilde{g}_k\, d\mu \right)^{q/t}\leq (C')^{q/t}[\mu(K)]^{q(1-t/p)/t}\sum_{k\in\mathbb{Z}}\left(\int_{K'} g_k^{p}\,d\mu\right)^{q/p}<\infty.
$$
Thus, for $C_E:=C_K^q (C_K'\tilde{C}C')^{q/t}[\mu(K)]^{q(1-t/p)/t}\sum_{k\in\mathbb{Z}}\left(\int_{K'} g_k^{p}\,d\mu\right)^{q/p}<\infty$, we have 
$$
\sum\limits_{i\in I} |f|_{\alpha,B_i}^q \leq 
C_K^q (C_K'\tilde{C})^{q/t} 
\sum_{i\in I} \sum_{k\in\mathbb{Z}}\left( \int_{\eps B_i} \tilde{g}_k\,d\mu\right)^{q/t}  \leq C_E,
$$
which implies $f\in CH^{q,s-Q/t}$ for all $t\in (Q/s,p)$. As such, it follows from Theorem~\ref{thm:main_Holder} that
$$
\dim_\theta f(E) \leq\frac{q d_E(\theta)}{(s-Q/t)q+d_E(\theta)}
$$ for all $t\in (Q/s,p)$, which, by letting for $t\rightarrow p$, implies 
$$
\dim_\theta f(E) \leq\frac{q d_E(\theta)}{(s-Q/p)q+d_E(\theta)}.
$$
This completes the proof of Theorem~\ref{thm:FractionalIntDim}.
\end{proof}

\begin{remark}\label{rem: local true}
    A closer study of the proof of Theorem~\ref{thm:FractionalIntDim} above reveals that essentially the finiteness of the respective semi-norms was only utilized when restricted to the compact set $K$. Thus, both the inclusion in the compactly H\"older class and the dimension bounds are in fact true even for continuous mappings in the local Haj\l{}asz--Triebel--Lizorkin and the local Haj\l{}asz--Besov classes.
\end{remark}
    
\section{Hausdorff and Minkowski dimension distortion}\label{sec: dimH et al}
    We first note that $X$ supporting a $Q$-Poincar\'e inequality implies that $X$ is connected \cite[Proposition~8.1.6]{Juha-Jeremy-etc-book} and, thus, uniformly perfect. As a result, Corollary~\ref{cor:QS_Dim} (i), (ii) are direct implications of \cite[Theorem~1.2]{Chron_comp_holder_Minkowski} and \cite[Theorem~9.3]{HeinKoskQCbegin}, respectively (see Remark~\ref{rem: old def}). See also the proof of \cite[Corollary~1.3]{Chron_comp_holder_Minkowski} for more details.
    
    Suppose $Q>1$ and $(X,d,\mu)$ is a proper, $Q$-homogeneous metric measure space, and $(Y,d_Y)$ is arbitrary.

    \begin{proof}[Proof of Theorem~\ref{thm: dimH dimB}]
        (i): Let $E\subset X$ non-empty with $d_E=\dim_H E<Q$. Due to Theorem~\ref{thm:FractionalIntDim}, it is enough to prove how a $(p,\alpha)$-compactly H\"older mapping $f:X\rightarrow Y$ distorts the Hausdorff dimension of $E$.

        Let $d>d_E$ and $D:=\tfrac{pd}{\alpha p+d}$. By the formulation of the Hausdorff dimension using the dyadic cube systems of Theorem~\ref{thm:Dydadic} (see \cite[Theorem~1.1 (i)]{chron_metric_dyadic_dim}), we have that for every $\varepsilon>0$ and every $\delta\in (0,\eps)$ there is a cover of $E$ by dyadic cubes $\{Q_i^{k_i}\}_{i\in I}$, with $b^{k_i}\leq \delta$ for all $i\in I$, such that
        \begin{equation}\label{eq: dH cover_sum}
            \sum_{i\in I} |Q_i^{k_i}| < \varepsilon.
        \end{equation} By Theorem~\ref{thm:Dydadic} (iii) and \eqref{eq: CH-inequality}, we have
        $$
        |f(Q_i^{k_i})|\leq |f(B(Q_i^{k_i}))|\leq C_i|B(Q_i^{k_i})|^\alpha\lesssim C_ib^{k_i \alpha}.
        $$ This implies that
        \begin{equation}\label{eq: d_H images small}
            |f(Q_i^{k_i})|\lesssim C_E \delta^\alpha,
        \end{equation} and that
        $$
        \sum_{i\in I}|f(Q_i^{k_i})|^{D}\lesssim \sum_{i\in I}C_i^Db^{Dk_i \alpha}.
        $$ By an application of H\"older's inequality for $p/D>1$, and noting that $D\alpha p(p-D)^{-1}=d$, similarly to the proof of Theorem~\ref{thm:main_Holder}, we have
        $$
        \sum_{i\in I}|f(Q_i^{k_i})|^{D}\lesssim C_E^{D/p} \left(\sum_{i\in I} b^{k_i d}\right)^{\tfrac{D-p}{p}}\lesssim \left(\sum_{i\in I} |Q_i^{k_i}|^d\right)^{\tfrac{D-p}{p}}.
        $$ By \eqref{eq: dH cover_sum} and \eqref{eq: d_H images small}, the above relation implies that $\{f(Q_i^{k_i})\}_{i\in I}$ is an appropriate cover of $f(E)$ to yield $\dim_H f(E)\leq D$. Letting $d\rightarrow d_E$ finishes the proof. 

        (ii): This follows directly by \cite[Theorem~1.1]{Chron_comp_holder_Minkowski} and Theorem~\ref{thm:FractionalIntDim}. 
    \end{proof}

    \begin{remark}
        It should be noted that in the case where $X,Y$ are uniformly perfect, letting $\theta\rightarrow 1$ in \eqref{eq:CHBoxdistortion} yields Theorem~\ref{thm: dimH dimB} (ii), without the use of \cite[Theorem~1.1]{Chron_comp_holder_Minkowski}. For Theorem~\ref{thm: dimH dimB} (i), although the proof is a very simplified case of that of Theorem~\ref{thm:main_Holder} for $\theta=0$,  simply letting $\theta\rightarrow 0$ in \eqref{eq:CHBoxdistortion} does not generally yield the desired inequality. This is due to the intermediate dimensions not necessarily being continuous at $\theta=0$ (see for instance \cite{Intermediate_dim_introduction}).
    \end{remark}


\end{document}